\documentclass[reqno,10pt,a4paper]{amsart}
\usepackage{amsmath}
\usepackage{amssymb}
\usepackage{amsbsy}
\usepackage{latexsym}
\usepackage{gensymb} 
\usepackage{mathtools} 

\usepackage{epsf}
\usepackage{graphics}
\renewcommand{\includegraphics}{\epsfbox}

\usepackage[usenames]{color}
\usepackage[dvips]{rotating}

\usepackage{pgfplots}
\usepgfplotslibrary{polar}
\pgfplotsset{compat=1.7}
\usepackage{tikz}
\usetikzlibrary{angles,calc,decorations.markings,arrows,backgrounds}
\usetikzlibrary{decorations.pathreplacing}
\usepackage{tikz-cd}
\usepackage{young}
\usepackage{caption}
\usepackage{subcaption}
\usepackage{floatrow}
\DeclareMathOperator{\Hom}{Hom}

\renewcommand{\le}{\leqslant}

\newcommand{\C}{\mathbb{C}}

\newcommand{\R}{\mathbb{R}}

\newcommand{\Z}{\mathbb{Z}}

\DeclareMathOperator{\supp}{supp}
\DeclareMathOperator{\Mod}{mod}
\DeclareMathOperator{\ch}{ch}

\newcommand{\GL}{\mathrm{GL}}

\newcommand{\TL}{\mathrm{TL}}

\newcommand{\sym}{\mathfrak{S}}

\DeclareMathOperator{\cha}{char}

\DeclareMathOperator{\Br}{Br}
\DeclareMathOperator{\Span}{Span}
\newcommand{\Sym}{\mathfrak{S}}

\begin{document}
\theoremstyle{plain}
\newtheorem{thm}{Theorem}[section]
\newtheorem{propn}[thm]{Proposition}
\newtheorem{cor}[thm]{Corollary}
\newtheorem{clm}[thm]{Claim}
\newtheorem{lem}[thm]{Lemma}
\newtheorem{conj}[thm]{Conjecture}
\theoremstyle{definition}
\newtheorem{defn}[thm]{Definition}
\newtheorem{rem}[thm]{Remark}
\newtheorem{exmp}[thm]{Example}

\title{On semi-simplicity of KMY algebras}
\author{N. M. Alraddadi \and A. E. Parker}
\address{
Mathematics Department\\ Al-Lith University College\\ Umm Al-Qura
University\\ Al-Lith\\ Saudi Arabia}

\address{Department of Mathematics \\ University of Leeds \\ Leeds,
  LS2 9JT \\ UK}
\email{nmraddadi@uqu.edu.sa}
\email{a.e.parker@leeds.ac.uk}



\begin{abstract}
We study the algebra, $J_{l,n}(\delta)$, introduced by Kadar-Martin-Yu
in  \cite{kadar}.
We show that these algebras are iterated inflation algebras.
We give a set of generators for $J_{l,n}(\delta)$.
We show that this algebra satisifies the CMPX axiomatic framework in
 \cite{cox2006representation} when the base field is the complex
numbers and $\delta \ne 0$.
We show that $J_{l,n}(\delta)$ is semisimple if $\delta$ is complex
and not real. 
\end{abstract}

\maketitle

\section*{Introduction}\label{intro}

This paper continues the study of a certain family of algebras that
live inside the Brauer algebra, thought of diagrammatically.
These algebras were introduced by Kadar, Martin and Yu in
\cite{kadar}.
We will introduce these algebras with less formality than that
used in \cite{kadar} as once certain crucial properties are
established (e.g. that heights cannot increase upon multiplication)
the strict formulism of \cite{kadar} is more of a hindrance to
understanding. We give a rather more seminar style introduction to
these algebras with the knowledge that should more formalism be
needed, it is all there in \cite{kadar} for the interested reader.

The bulk of this paper is taken from the first author's PhD thesis,
the second author was the first author's main supervisor. We have
added one result not in the thesis, that KMY algebras are iterated
inflation algebras, and generalised another, extending the result
about semi-simplicity to all $l$.

There are four new main results in this paper. The first,
Theorem~\ref{thm:iterate},  is that the KMY
algebras are iterated inflation algebras, as is true for Brauer
algebras.

The second, section~\ref{cmpx}, is that when the ring (field) of
coefficients is $\C$ then  
the KMY algebras satisfy
the axiomatic framework for a tower of recollement introduced in \cite{cox2006representation}. Much of the
groundwork for this result was laid down in the KMY paper,
\cite{kadar}.

The third, Theorem~\ref{thm:genset}, gives a generating set for the
KMY algebra.

The fourth, Theorem~\ref{thm:semisimple}, gives a criterion for the
KMY algebra to be semi-simple, when the ring (field) of coeefficients
is $\C$, namely that the parameter $\delta$ is not real. 

\section{Review}\label{review}

A \emph{pair partition} of a set  is a set partition of the set where each part
has size two (i.e. each element is paired with another).
A \emph{Brauer diagram} is a picture/diagram of a pair partition of
    $$
    \{1,2,\ldots, n\} \cup \{1',2',\ldots,n'\}
    $$
which is a rectangle (also called a \emph{frame}) with $n$ (labelled) vertices on the top and $n$
(labelled) vertices on
the bottom that are joined by a line,  (which is often smooth and we
allow to be isotopically deformed),  if the labels on the end vertices
form a pair in the
pair partition. The labels on the vertices are usually omitted if it is clear.
The important information here is the pairs in the pair partition,
represented visually in the diagram, rather than the details of the
diagram itself.
E.g. for  $n=5$,  the pair partition $\{\{1,2\},\{2',3\}, \{1,4\},
\{4',5\},\{3',5'\}\}$ can be represented by the diagram:
$$
\begin{picture}(0,0)%
\includegraphics{nis5fig1.pstex}%
\end{picture}%
\setlength{\unitlength}{3947sp}%
\begingroup\makeatletter\ifx\SetFigFont\undefined%
\gdef\SetFigFont#1#2#3#4#5{%
  \reset@font\fontsize{#1}{#2pt}%
  \fontfamily{#3}\fontseries{#4}\fontshape{#5}%
  \selectfont}%
\fi\endgroup%
\begin{picture}(2274,1378)(289,-641)
\put(2251,614){\makebox(0,0)[lb]{\smash{{\SetFigFont{10}{12.0}{\rmdefault}{\mddefault}{\updefault}{\color[rgb]{0,0,0}$5$}%
}}}}
\put(451,-586){\makebox(0,0)[lb]{\smash{{\SetFigFont{10}{12.0}{\rmdefault}{\mddefault}{\updefault}{\color[rgb]{0,0,0}$1'$}%
}}}}
\put(901,-586){\makebox(0,0)[lb]{\smash{{\SetFigFont{10}{12.0}{\rmdefault}{\mddefault}{\updefault}{\color[rgb]{0,0,0}$2'$}%
}}}}
\put(1351,-586){\makebox(0,0)[lb]{\smash{{\SetFigFont{10}{12.0}{\rmdefault}{\mddefault}{\updefault}{\color[rgb]{0,0,0}$3'$}%
}}}}
\put(1801,-586){\makebox(0,0)[lb]{\smash{{\SetFigFont{10}{12.0}{\rmdefault}{\mddefault}{\updefault}{\color[rgb]{0,0,0}$4'$}%
}}}}
\put(2251,-586){\makebox(0,0)[lb]{\smash{{\SetFigFont{10}{12.0}{\rmdefault}{\mddefault}{\updefault}{\color[rgb]{0,0,0}$5'$}%
}}}}
\put(451,614){\makebox(0,0)[lb]{\smash{{\SetFigFont{10}{12.0}{\rmdefault}{\mddefault}{\updefault}{\color[rgb]{0,0,0}$1$}%
}}}}
\put(901,614){\makebox(0,0)[lb]{\smash{{\SetFigFont{10}{12.0}{\rmdefault}{\mddefault}{\updefault}{\color[rgb]{0,0,0}$2$}%
}}}}
\put(1351,614){\makebox(0,0)[lb]{\smash{{\SetFigFont{10}{12.0}{\rmdefault}{\mddefault}{\updefault}{\color[rgb]{0,0,0}$3$}%
}}}}
\put(1801,614){\makebox(0,0)[lb]{\smash{{\SetFigFont{10}{12.0}{\rmdefault}{\mddefault}{\updefault}{\color[rgb]{0,0,0}$4$}%
}}}}
\end{picture}%

\ \raisebox{40pt}{.}
$$

We fix $n$ and use all possible diagrams as a basis for a vector space
over a field (or ring) $K$.
To make an algebra we stack diagrams on top of each other to multiply:
E.g.
$$
\raisebox{-0.45\height}{\includegraphics{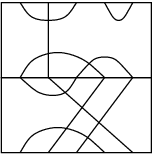}}
\quad = 
\quad \raisebox{-0.4\height}{\includegraphics{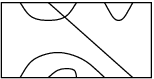}}
$$
We might get loops:
$$
\raisebox{-0.45\height}{\includegraphics{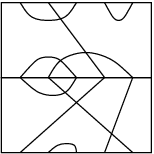}}
\quad = \quad \delta
\quad \raisebox{-0.4\height}{\includegraphics{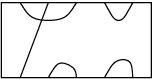}}
$$
The power of the parameter $\delta\in K$ records the number of loops
removed.
$$
\raisebox{-10pt}{\includegraphics{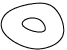}}\rightsquigarrow \delta^2
\qquad
\raisebox{-5pt}{\includegraphics{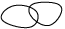}} \rightsquigarrow \delta^2
$$

$\raisebox{-10pt}{\includegraphics{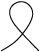}}$
A twist gets untwisted --- we don't record untangling of twists.

The Brauer algebra has the symmetric group algebra as a sub-algebra via
permutation diagrams.
E.g. for $n=5$:
$$
\raisebox{-15pt}{\includegraphics{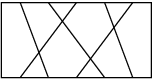}}
\leftrightarrow \begin{pmatrix}1&2&3&4&5\\ 2&4&1&5&3\end{pmatrix}
$$

We call a line \emph{propagating} if it goes from top to bottom.
Otherwise it's called an \emph{arc},
or sometimes a cap $\cap$ or a cup $\cup$.

With this diagram realisation we have
$$
K\mathfrak{S}_n \hookrightarrow \mathrm{Br}_n
$$
where $\mathfrak{S}_n$ is the symmetric group on $\{1,2,\ldots,n\}$
and
$\mathrm{Br}_n$ is the Brauer algebra on $n$.

Note:
$\raisebox{-5pt}{\includegraphics{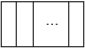}}$
is the identity permutation and is the same as the identity diagram in
$\mathrm{Br}_n$.

We note that the representation theory of $\mathrm{Br}_n$ is at least as
complicated as that for $K\mathfrak{S}_n$ and indeed it's unknown for
$\cha K =p <n$.

There is another algebra inside $\Br_n$, namely the
\emph{Temperley-Lieb algebra}, $\TL_n$.

Take a \emph{planar} representation of a diagram.
If this can be deformed in the plane to have \emph{no} crossings
then it's a \emph{Temperley-Lieb diagram}.
E.g. 
$$\raisebox{-15pt}{\includegraphics{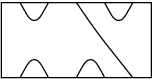}}
\qquad \mbox{or} \qquad
\raisebox{-15pt}{\includegraphics{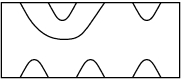}}
$$
The representation theory of $\TL_n$ is \emph{very} well
understood. (Analogous to $q$-$\GL_2$.)

Are there algebras in between that might be easier to understand?
Well yes.
Take a Brauer diagram.
Define \emph{height} as follows.
Take the following diagram:
$$
\includegraphics{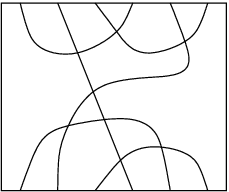}
$$
We first consider the interior connected regions (alcoves) in the diagram and
number from $0$ from the left, increasing the number by one if we
cross a line.
$$
\includegraphics{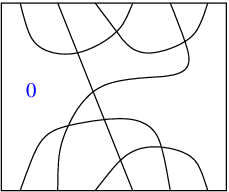}
\quad \raisebox{40pt}{$\to$} \quad
\includegraphics{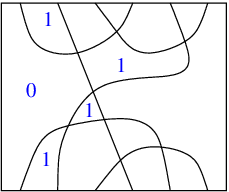}
\quad \raisebox{40pt}{$\to$} \quad
\includegraphics{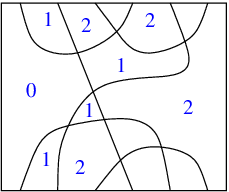}
$$
$$
\quad \raisebox{40pt}{$\to$} \quad
\includegraphics{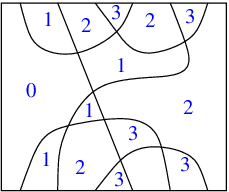}
\quad \raisebox{40pt}{$\to$} \quad
\includegraphics{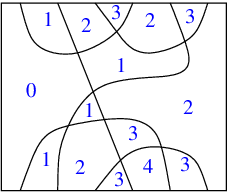}
$$
The number in each alcove therefore records the minimum number of
lines that are crossed by any path joining that alcove to the left
most side of the rectangle (frame) for the diagram.

We circle all the \emph{crossing points}, that is the points where two lines
intersect. (Remembering that we do not allow diagrams where more than
2 lines intersect at a point.) We then label each crossing point with
the smallest number of the alcove that the circle around the crossing
point intersects.
$$
\quad \raisebox{40pt}{$\to$} \quad
\includegraphics{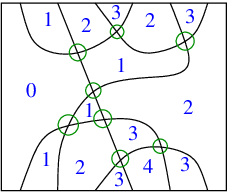}
\quad \raisebox{40pt}{$\to$} \quad
\includegraphics{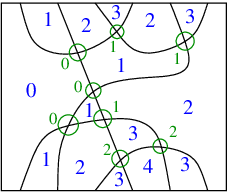}
$$
The \emph{height} of this particular diagram is then the maximum of all the
heights of the crossing points.
The \emph{height} of a diagram (as a pair partition) is the minimum of the
heights of all possible ways of representing the pair partition with a
planar diagram.

Temperley-Lieb diagrams which have no crossings, are defined to have height $-1$.

That this definition is possible is covered in \cite{kadar}.

Let $h(d)$ be the height of a diagram $d$.
In \cite{kadar} they prove:
$$
h(d_1d_2) \le \max\{ h(d_1), h(d_2)\}
$$
where $d_1$ and $d_2$ are Brauer diagrams.

That is the  height can go \emph{down} but never go \emph{up}  when we
multiply.

Thus diagrams of height at most $l$ give a basis for a subalgebra,
$J_{l,n}(\delta)$, of
$\Br_n$.

This algebra was first defined in \cite{kadar} and we will henceforth call it
\emph{the KMY-algebra}.

We have 
$$
J_{-1,n}(\delta) = \TL_n
$$
$$
J_{n-2,n}(\delta) = \Br_n
$$
as Brauer diagrams have maximum height $n-2$.

Many of the properties of $\Br_n$ are also true for 
$J_{l,n}(\delta)$.

Let $\#(d)$ be the number of propagating lines in $d$.
Then: $$ \#(d_1d_2) \le \min \{\#(d_1), \#(d_2)\} .$$

We define ideals,  $I_m = \langle d \in J_{l,n}(\delta) \mid \#(d) \le m \rangle$,
and these ideals filter the algebra.
For $n$ even:
$$I_0 \le I_2 \le \cdots \le I_{n-2} \le I_n =
J_{l,n}(\delta).$$
And for $n$ odd:
$$ I_1 \le I_3 \le \cdots \le I_{n-2} \le I_n =
J_{l,n}(\delta). $$

One result of \cite{kadar} is that the KMY algebra, $J_{l,n}(\delta)$,
is \emph{cellular} with a similar cell basis as for
$\Br_n$.
This uses half diagrams and bases for Specht modules for $K\Sym_r$,
where 
 $r\le n$.
Our labelling poset for the cell modules consists of pairs $(r,
\lambda)$ where $r$ is the number of propagating lines in the half
diagram basis for the cell module, and $\lambda$ is a partition of
either $l+2$ or $r$ depending on which is smaller.

 Half diagrams are formed from cutting the propagating lines.
$$
\includegraphics{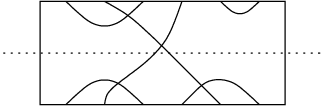}
$$
We ignore crossings in the propagating lines:
$$
\includegraphics{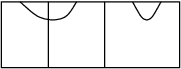}
$$

Let $S^{\lambda}$ be a Specht module for $K\Sym_m$, with
$\lambda$ a partition of $m$, $\lambda \vdash m$.

We form a basis for the cell module $\Delta_{l,n}(r,\lambda)$ by taking
half diagrams with $r$ propagating lines, height at most $l$, $\lambda \vdash
\min\{r,l+2\}$ and a (fixed) basis for $S^\lambda$.

NB:  We are agnostic
about what basis to take for the Specht module. We treat it rather
like a black box. In the applications in this paper, the Specht
modules will all be simple.

Let $v$ be  in a basis for $S^\lambda$. 
Elements of the half diagram basis look like: 
$$
\includegraphics{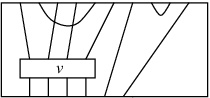}
$$
We get an action of $J_{l,n}(\delta)$  like so:
$$
\includegraphics{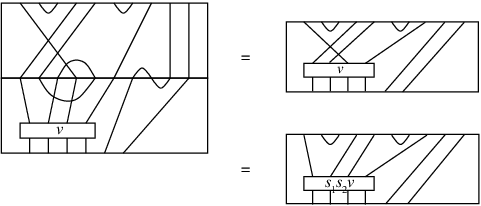}
$$
where $s_1s_2=(1,2)(2,3)$ is a permutation.
That is any crossings in the propagating lines (which mean the written
diagram is not in the basis) get converted into a permutiation on
$\min\{r, l+2\}$ lines which is then pushed into the box representing
the element of the Specht module. We then act with this permutation in
the Specht module, getting ultimately a linear combination of basis elements. 

An arc acts as zero:
$\raisebox{-5pt}{\includegraphics{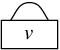}}=0$. (Kills off propagating lines.)

E.g. for $n=4$, $l=0$ we get
$$
\Delta_{0,4}(0,\emptyset) = \left\langle
\raisebox{3pt}{\includegraphics{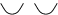}},
\raisebox{1.5pt}{\includegraphics{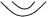}},
\raisebox{3pt}{\includegraphics{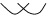}}
\right\rangle
$$
$$
\Delta_{0,4}(2,\lambda) = \left\langle
\raisebox{-5pt}{\includegraphics{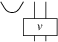}},
\raisebox{-5pt}{\includegraphics{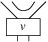}},
\raisebox{-5pt}{\includegraphics{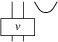}},
\raisebox{-5pt}{\includegraphics{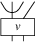}}
\right\rangle
$$
where $\lambda \vdash 2$ and $S^\lambda = \Span\{v\}$.
$$
\Delta_{0,4}(4,\lambda) = \left\langle
\raisebox{-5pt}{\includegraphics{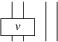}},
\right\rangle
$$
where $\lambda \vdash 2$  and $S^\lambda = \Span\{v\}$.

We may use these bases to define Gram matrices and determinants as
usual for cellular algebras.

\section{KMY algebras are iterated inflation algebras}

\begin{thm}\label{thm:iterate}
The algebra $J_{l,n}(\delta)$ is an iterated inflation of group
algebras of symmetric groups.

I.e.,  the algebra $J_{l,n}(\delta)$ has a filtration by ideals:
n even:
$$
I_0 \le I_2 \le \cdots \le I_{n-2} \le I_n = J_{l,n}(\delta).
$$
n odd:
$$
I_1 \le I_3 \le \cdots \le I_{n-2} \le I_n = J_{l,n}(\delta).
$$
where
$$ I_{m}/I_{m-2} \cong 
V_{m}\otimes  k\sym_{m_l} \otimes V_{m}^\circ  
$$ 
as a $K$-module, $m_l = \min\{l+2, m\}$,
$V_m$ is the free $K$-module with basis
the upper half diagrams on $m$ propagating lines,
and $V_m^\circ$ is the free $K$-module with basis
the lower half diagrams on $m$ propagating lines.
\end{thm}

We may take
 $I_m = \langle d \in J_{l,n}(\delta) \mid \#(d) \le m \rangle$ 

The construction of the cell module basis essentially gives us this
result and it follows exactly as for the Brauer algebra.

\section{KMY algebras satisfy the CMPX axioms}\label{cmpx}

For this section and much of the remainder of the paper we will set
$K=\C$, the complex numbers. The main sticking point for generalising
to other $K$ is axiom (A2). This could be generalised if we can ensure
that the symmetric group algebra $K\sym_n$ was semisimple for that
$K$. I.e. we have to assume that $\ch K > l$, and this is indeed
possible, in a way that it's not for the Brauer algebra, where we
would need $\ch K > n$ with $n$ increasing up the chain. We leave the
details to the interested reader.

In 2006, Cox, Martin, Parker and Xi \cite{cox2006representation},
introduce an axiomatic framework for studying the representation
theory of tower of algebras. Each family of algebras that satisfy
these axioms is called a \emph{tower of recollement}.

\subsection{Axioms for towers of recollement}

Let $e$ be an idempotent of an algebra $A.$
\begin{thm}[\cite{green1980polynomial}]
 Let $\{L(\lambda)\mid \lambda \in \Lambda\}$ be a full set of simple
 $A$-modules, and set $\Lambda^e=\{\lambda \in \Lambda \mid
 eL(\lambda)\neq 0\}.$ Then $\{eL(\lambda)\mid \lambda\in \Lambda^e
 \}$ is a full set of simple $eAe$-modules. Further, the simple
 modules $L(\lambda)$ with $\lambda\in \Lambda\setminus\Lambda^e$ are
 a full set of simple $A/AeA$-modules.
\end{thm}
Let $A_n$ with $n\geq 0$ be a family of finite-dimensional algebras with idempotents $e_n$ in $A_n$.
\begin{defn}[{\cite[Section 1]{cox2006representation}}]
Let $\Lambda_n$ be the indexing set for the simple $A_n$-modules and
$\Lambda^n$ be the indexing set for the simple
$A_n/A_ne_nA_n$-modules.
\end{defn}
\begin{defn}[{\cite[Section 1]{cox2006representation}}]
For $m, n \in \mathbb{N}$ with $m-n$ even we set
$\Lambda_m^n=\Lambda^n$ regarded as a subset of $\Lambda_m$ if
$m\geq n$, and $\Lambda_m^n=\emptyset$ otherwise.
\end{defn}
\begin{defn}[{\cite[Section 1]{cox2006representation}}]
A $\Delta_n$-filtration is defined to be a filtration of a module such
that successive quotients are isomorphic to $\Delta_n(\mu)$ for some
\end{defn}
\begin{defn}[{\cite[Section 1]{cox2006representation}}]
If a module $M$ in $A_n$-$\Mod$ has a $\Delta_n$-filtration we define
the support of $M$, denoted $\supp_n(M)$, to be the set of
labels $\lambda$ for which $\Delta_n(\lambda)$ occurs in this
filtration.
\end{defn}
Now we present the axioms of a tower of recollement as in \cite{cox2006representation}.
\begin{enumerate}
\item [(A1)] For each $n\geq 2$ we have an isomorphism
  $$
  \Phi_n:A_{n-2}\longrightarrow e_nA_ne_n.
  $$
  We have
  $$
  \Lambda_n=\Lambda^n\sqcup\Lambda_{n-2}.
  $$
\end{enumerate}
Set $e_{n,0}=1$ in $A_n$ and $1\leq i \leq \frac{n}2$ define new
idempotents in $A_n$ by setting $e_{n,i}= \Phi_n(e_{n-2,i-1}).$
\begin{enumerate}
 \item [(A2)]For each $n\geq 0$ the algebra $A_n$ is quasi-hereditary
   with heredity chain of the form
   $$
  0\subset \cdots \subset A_ne_{n,i}A_n \subset \cdots \subset A_ne_{n,0}A_n=A_n.
  $$
   As $A_n$ is quasi-hereditary, there is for each $\lambda \in
   \Lambda_n$ a standard module $\Delta_n(\lambda)$ with simple head
   $L_n(\lambda)$. If we set $\lambda \leq \mu$ if either
   $\lambda=\mu$ or $\lambda\in \Lambda_n^r$ and $\mu\in \Lambda_n^s$
   with $r>s$, then all other composition factors of
   $\Delta_n(\lambda)$ are labelled by weights $\mu$ with
   $\mu<\lambda$. For $\lambda\in \Lambda_n^n$, we have that
   $\Delta_n(\lambda)\cong L_n(\lambda),$ and that this is just the
   lift of a simple module for the quotient algebra $A_n/A_ne_nA_n$.
\end{enumerate}
 We recall the localisation functor, $\mathcal{F}_n$,
 \begin{align*}
 \mathcal{F}_n:A\text{-}\Mod&\rightarrow eAe\text{-}\Mod\\
 M&\rightarrow eM
 \end{align*}
and the globalisation functor, $\mathcal{G}_n$,
 \begin{align*}
 \mathcal{G}_n:eAe\text{-}\Mod&\rightarrow A\text{-}\Mod   \\
 N&\rightarrow Ae\otimes_{eAe}N. 
 \end{align*}
From \cite[Proposition 3]{martin2004virtual} we have
$$
\mathcal{G}_n(\Delta_n(\lambda))\cong \Delta_{n+2}(\lambda).
$$
$$
          \mathcal{F}_n(\Delta_n(\lambda))\cong \begin{cases}
\Delta_{n-2}(\lambda), & \mbox{if } \lambda \in \Lambda_{n-2} \\
0, & \mbox{if } \lambda \in \Lambda^n.
\end{cases}
$$
\begin{enumerate}
       \item[(A3)] For each $n\geq 0$ the algebra $A_n$ can be identified with a subalgebra of $A_{n+1}$.
       \item[(A4)] For all $n\geq 1$ we have $A_ne_n\cong A_{n-1}$ as a left $A_{n-1}$-, right $A_{n-2}$-bimodule.
where we view $A_ne_n$ as a left $A_{n-1}$ module via the embedding in
(A3), and as a right $A_{n-2}$ via the isomorphism from (A1) and we
view $A_{n-1}$ as a right $A_{n-2}$-module via the embedding from (A3).
       \item[(A5)] For each $\lambda \in \Lambda_n^m$ we have that the
         restriction from $A_n$ to $A_{n-1}$ (via the embedding in
         (A3)), $\Delta_n(\lambda)\downarrow$, has a $\Delta$-filtration and
       $$
         \supp(\Delta_n(\lambda)\downarrow)\subseteq \Lambda_{n-1}^{m-1}\sqcup \Lambda_{n-1}^{m+1}.
       $$
     \item[(A6)]
       For each $\lambda \in \Lambda_n^n$ there exist $\mu\in
       \Lambda_{n+1}^{n-1}$ such that $\lambda \in
       \supp(\Delta_{n+1}(\mu)\downarrow).$
\end{enumerate}
If axioms (A1) to (A6) are satisfied then we have the following
theorem for finding homomorphisms between two cell modules and a
criterion for semi-simplicity.
\begin{thm} [{\cite[theorem 1.1]{cox2006representation}}]    \label{thm:thm1.1}
Let $\lambda\in \Lambda_n^m$ and $\mu\in \Lambda_n^t$. 
\begin{enumerate}
  \item[(i)]  For all pairs of weights $\lambda\in \Lambda_n^m$ and
    $\mu\in \Lambda_n^t$ we have
    $$
  \Hom(\Delta_n(\lambda),\Delta_n(\mu))\cong\begin{cases}
  \Hom(\Delta_m(\lambda),\Delta_m(\mu)), & \mbox{if } t\leq m \\
  0, & \mbox{otherwise}.
  \end{cases}
 $$
  \item[(ii)]   Suppose that for all $n\geq 0$ and pairs of weights
    $\lambda\in \Lambda_n^n$ and $\mu\in \Lambda_n^{n-2}$ we have
    $$
   \Hom\bigl(\Delta_n(\lambda),\Delta_n(\mu)\bigr)=0.
   $$
\end{enumerate}
Then each of the algebras $A_n$ is semisimple.
\end{thm}
\subsection{Algebra $J_{l,n}(\delta)$ satisfies the axioms (A1) to (A4)} \label{second section in ch 4}
In the following we can see that $J_{l,n}(\delta)$ satisfies the axioms (A1) to (A4).
Firstly we define the idempotent $e_n$ in algebra $J_{l,n}(\delta).$
\begin{defn}
For $n\geq2$ and $\delta\neq 0,$ we define the idempotent element $e_n\in J_{l,n}(\delta)$ to be the diagram
\begin{center}
 \begin{tikzpicture}
    \draw (0,0)--(0,1.5) node at (-1.9,.75){$e_n ={\frac{1}{\delta}}(1_{n-2}\otimes U )={\frac{1}{\delta}}$};
    \draw (0,1.5)--(4,1.5);
    \draw (4,1.5)--(4,0) node at (1.75,.75){$\cdots$};
    \draw (0,0)--(4,0) ;
\draw (.5,1.5)--(.5,0);
\draw (1,1.5)--(1,0);
\draw  (2.5,1.5)--(2.5,0);
\draw (3,1.5)..controls(3,1.1) and (3.5,1.1)..(3.5,1.5) ;
\draw (3,0)..controls(3,.4) and (3.5,.4)..(3.5,0);
\draw [decorate,decoration={brace,amplitude=5pt,mirror,raise=.5ex}]
  (.5,-.1) -- (2.5,-.1) node[midway,yshift=-1em]{$n-2$};
  \end{tikzpicture}
\end{center}
where  there are $n-2$ propagating lines in the diagram for $e_n$, $U=\begin{tikzpicture}
\draw (3.5,.8)..controls(3.5,.4) and (4,.4)..(4,.8) ;
\draw (3.5,0)..controls(3.5,.4) and (4,.4)..(4,0);
  \end{tikzpicture}$ and $\otimes$ denotes side-by-side concatenation.
\end{defn}
Now we prove that the algebra $J_{l,n}(\delta)$ satisfies all the
axioms in the CMPX framework when our base field is the complex
numbers and $\delta \ne 0$.
\begin{propn}[{\cite[Proposition 4.1]{kadar}}]\label{axiom1}
For $n>2$, we have $$
                     J_{l,n-2}\cong e_nJ_{l,n}e_n
                   $$
                   so $J_{l,n}(\delta)$ satisfies axiom (A1).
 \end{propn}
 For $n\geq 2t$, where $t\in \mathbb{N}$, we define
$e_{n,t}=1_{n-2t}\otimes U^{\otimes t}$ and when $\delta \in
\mathbb{C}^{\times}$ we define $e_{n,t}^\prime=\delta^{-t}e_{n,t}.$ We
can see that $e_{n,t}\in J_{-1}(n, n-2t, n),$ i.e. $e_{n,t}$ has
no crossings and is a Temperley-Lieb diagram.
\begin{thm}[{\cite[Theorem 4.12]{kadar}}]
  If $\delta$ invertible in $\mathbb{C}$ then $J_{l,n}$ is
  quasi-hereditary, with heredity chain given by $(1, e_{n,1}^\prime,
  e_{n,2}^\prime, \ldots)$. So $J_{l,n}(\delta)$ satisfies axiom (A2).
\end{thm}
\begin{thm} \label{Theorem of AxiomA3}
For all $n\geq1$, the algebra $J_{l,n}(\delta)$ is a subalgebra of
$J_{l,n+1}(\delta)$. So $J_{l,n}(\delta)$ satisfies axiom (A3).
\end{thm}
\begin{proof}
  We define an inclusion map
  $i_n:J_{l,n}(\delta)\xhookrightarrow{}J_{l,n+1}(\delta)$ by adding a
  vertical line to the right of the diagram $d\in J_{l,n}(\delta)$ and
  then extended by linearity to all of $J_{l,n}(\delta)$. This
  inclusion sends the identity diagram in $J_{l,n}(\delta)$ to the
  identity in $J_{l,n+1}(\delta)$.

 Assume that $d_1$, $d_2\in J_{l,n}(\delta)$, we have
 $$
  i(d_1d_2)=i(d_1)i(d_2).
 $$
 This means if we multiply two diagrams $d_1$ and $d_2$, then add a
 vertical line on the right of the multiplication $d_1d_2$ is the same
 when we add a vertical line on the right of $d_1$ and $d_2$ then
 multiply them. Therefore the inclusion map $i$ is a
 homomorphism. Since different diagrams in $J_{l,n}(\delta)$ have
 different images in $J_{l,n+1}(\delta)$, the map $i$ is injective.
\end{proof}
\begin{propn}[{\cite[Lemma 6.2]{kadar}}]  For all $n\geq 1$ we have
  $J_{l,n}e_n\cong J_{l,n-1}$ as a left $J_{l.n-1}$-, right
  $J_{l,n-2}$-bimodule. So $J_{l,n}(\delta)$ satisfies axiom (A4).
\end{propn}

\subsection{The localisation and globalisation functors}\label{third section in ch4}

We use the previously definied idempotent $e_n$ to define functors.
The \emph{localisation} functor is given by
\begin{align*}
\mathcal{F}_n:J_{l,n}\text{-}\Mod&\rightarrow J_{l,n-2}\text{-}\Mod \\
M&\mapsto e_nM.
\end{align*}
The \emph{globalisation} functor is given by
\begin{align*}
\mathcal{G}_n:J_{l,n}\text{-}\Mod&\rightarrow J_{l,n+2}\text{-}\Mod\\
M&\mapsto J_{l,n}e_n\otimes_{e_nJ_{l,n}e_n}M.
\end{align*}

\subsection{Restriction of the cell module $\Delta_{l,n}(p,\lambda)$} \label{fourth section in ch4}

The restriction of $\Delta_{l,n}(p,\lambda)$ to $J_{l,n-1}(\delta)$
via the embedding $i_{n-1}$ in (A3) is denoted by $\Delta_{l,n}(p,\lambda)\downarrow$.
In the following we present restriction rules for the cell modules in
the algebra $J_{l,n}(\delta)$.
\begin{thm}[{\cite[Proposition 5.3]{kadar}}]\label{Proposition 5.3}
 Consider $i_{n-1}:J_{l,n-1}(\delta)\hookrightarrow J_{l,n}(\delta)$,
 denote the corresponding restriction of $J_{l,n}(\delta)$ to a
 $J_{l,n-1}(\delta)$ by $\Delta_{l,n}(p,\lambda)\downarrow$. Then we
 have for fixed $l$ the following exact sequences for the restriction
\begin{enumerate}
  \item[(i)] For $p<l+2$ we have a short exact sequence
  $$
    0\rightarrow\bigoplus_i\Delta_{l,n-1}(p-1,\lambda
    -e_i)\rightarrow\Delta_{l,n}(p,\lambda)\downarrow\rightarrow\bigoplus_i\Delta_{l,n-1}{(p+1,\lambda+e_i)}
    \rightarrow 0
  $$
  where the sums are over addable/removable boxes of the Young diagram
  and $\lambda-e_i$ denotes $\lambda$ with the $i$-th removable box
  removed and so on.
 \item[(ii)] For $p=l+2$
 $$
 0\rightarrow\bigoplus_i\Delta_{l,n-1}(p-1,\lambda
 -e_i)\rightarrow\Delta_{l,n}(p,\lambda)\downarrow
 \rightarrow\Delta_{l,n-1}{(p+1,\lambda)}\rightarrow
 0
 $$
 where the sums are over removable boxes of the Young diagram and
 $\lambda-e_i$ denotes $\lambda$ with the $i$-th removable box
 removed and so on.
 \item[(iii)] For $p>l+2$ $$0\rightarrow\Delta_{l,n-1}(p-1,\lambda)
      \rightarrow\Delta_{l,n}(p,\lambda)\downarrow\rightarrow\Delta_{l,n-1}{(p+1,\lambda)}\rightarrow0$$
\end{enumerate}
\end{thm}

\subsection{Proof of the axioms (A5) and (A6)}\label{fifth section in ch4}
\begin{propn} \label{proposition of axiom A5}
For each $(p,\lambda)\in \Lambda_{n,p_l}^p$ we have that
$\Delta_{l,n}(p,\lambda)\downarrow$ has a $\Delta$-filtration and
       $$
         \supp(\Delta_{l,n}(p,\lambda)\downarrow)\subseteq \Lambda_{{n-1,p_l}}^{p-1}\sqcup \Lambda_{{n-1,p_l}}^{p+1}.
         $$ So the algebra $J_{l,n}(\delta)$ satisfies axiom (A5).
 \end{propn}
       \begin{proof}
Let $(p,\lambda)\in \Lambda_{n,p_l}^p.$ By applying Theorem \ref{Proposition 5.3} we have\\
Case 1: If $p<l+2,$ we have \begin{align*}
\begin{split}
   \supp(\Delta_{l,n}(p,\lambda)\downarrow)&=\bigcup_i\{(p-1,\lambda-e_i)\}\cup
   \bigcup_i \{(p+1,\lambda+e_i)\} \\
     &\subseteq \Lambda_{{n-1,p_l}}^{p-1}\sqcup \Lambda_{{n-1,p_l}}^{p+1}
\end{split}
                            \end{align*}
Case 2: If $p=l+2,$ we have
\begin{align*}
\begin{split}
   \supp(\Delta_{l,n}(p,\lambda)\downarrow)&= \bigcup_i\{(p-1,\lambda-e_i)\}\cup \{(p+1,\lambda)\} \\
            &\subseteq \Lambda_{{n-1,p_l}}^{p-1}\sqcup \Lambda_{{n-1,p_l}}^{p+1} \end{split}
                            \end{align*}
Case 3: If $p>l+2,$ we have \begin{align*}\renewcommand{\qedsymbol}{$\square$}
\begin{split}
   \supp(\Delta_{l,n}(p,\lambda)\downarrow)&= \{(p-1,\lambda)\}\cup \{(p+1,\lambda)\} \\
            &\subseteq \Lambda_{{n-1,p_l}}^{p-1}\sqcup \Lambda_{{n-1,p_l}}^{p+1}\qedhere \end{split}
                            \end{align*}
 \end{proof}
\begin{propn}
For each $(p,\lambda)\in \Lambda_{n,p_l}^p$ there exists $(p^\prime,\mu)\in\Lambda_{n+1,p_l^\prime}^{p^\prime-1}$ such that
$$
  (p,\lambda)\in\supp(\Delta_{l,n+1}(p^\prime,\mu)\downarrow).
  $$
  I.e., the algebra $J_{l,n}(\delta)$ satisfies axiom (A6).
\end{propn}
\begin{proof}
Let $(p,\lambda)\in \Lambda_n^p$ and let $p^\prime > p$, we take $p^\prime=p+1.$ We have the following cases:\\
Case 1: If $p>l+2,$ we have $p^\prime=p+1>l+2.$ So $\lambda\vdash l+2$ and $\mu\vdash l+2.$\\
We have $\lambda=\mu$ and $p=p^\prime-1.$
$$(p,\lambda)=(p^\prime -1,\mu).$$
Then we have $(p,\lambda)\in\supp(\Delta_{l,n+1}(p^\prime,\mu)\downarrow).$\\\\
Case 2: If $p=l+2,$ we have $\lambda\vdash p=l+2$ and $\mu\vdash l+2=p.$ So
$$(p,\lambda)=(p^\prime -1,\mu).$$
Therefore we have $(p,\lambda)\in\supp(\Delta_{l,n+1}(p^\prime,\mu)\downarrow).$\\\\
Case 3: If $p<l+2,$ we have $\lambda\vdash p.$
$\lambda$ has an addable node $e_i$, either $\mu\vdash p^\prime, p^\prime<l+2$ or $\mu\vdash p^\prime=l+2.$ \\
$\mu$ has a removable node, let $\mu=\lambda+e_i.$\\
We have $(p,\lambda)=(p^\prime -1, \mu-e_i).$ Then $(p,\lambda)\in\supp(\Delta_{l,n+1}(p^\prime,\mu)\downarrow).$
\end{proof}
The algebra $J_{l,n}(\delta)$ satisfies axioms (A1) to (A6). We
restate Theorem 1.1 (theorem~\ref{thm:thm1.1} in this paper) from \cite{cox2006representation} for the context
of our algebra $J_{l,n}(\delta)$. It is this theorem that will allow
us to prove Theorem~\ref{thm:semisimple}.
\begin{thm} \label{J_{l,n}semisimple}
\begin{enumerate}
  \item[(i)]  For $l=\{-1, \ldots, n-2\}$ and all pairs of
    weight $(p,\lambda)\in \Lambda_{n,p_l}^p$ and $(p^\prime,\mu)\in
    \Lambda_{n,p^\prime_l}^{p^\prime}$ we have
    $$
   \Hom(\Delta_{l,n}(p,\lambda),\Delta_{l,n}(p^\prime,\mu))\cong\begin{cases}
   \Hom(\Delta_{l,p}(p,\lambda),\Delta_{l,p}(p^\prime,\mu)), & \mbox{if } p^\prime\leq p \\
   0, & \mbox{otherwise}.
   \end{cases}
$$
  \item[(ii)]  Suppose that for all $n\geq 0$ and pairs of
    weights $(n,\lambda)\in \Lambda_{n,p_l}^n$ and $(n-2,\mu)\in
    \Lambda_{n,p_l}^{n-2}$ we have $$
\Hom(\Delta_{l,n}(n,\lambda),\Delta_{l,n}(n-2,\mu))=0.
$$
\end{enumerate}
Then each of the algebras $J_{l,n}(\delta)$ is semisimple.
\end{thm}

\section{A generating set for the algebra $J_{l,n}(\delta)$}\label{generating setfor J}

In this section we give a generating set for the algebra
$J_{l,n}(\delta)$. We work over the ring (field) of coefficients $\C$
with $\delta$ as a parameter.  However, in this section the ring of
coefficients does not play a strong role, and the result holds more
generally.

The
first type of generator $u_i$, $1\leq i\leq n-1$ are all TemperleyLieb
diagrams (height $-1$) and generate the 
Temperley-Lieb algebra as a subalgebra of $J_{l,n}(\delta)$.
The second type are the crossings $s_m$, where $1\leq m \leq l+1$, and
generate the symmetric group algebra $\C\sym_{m}$  as a subalgebra of $J_{l,n}(\delta)$.
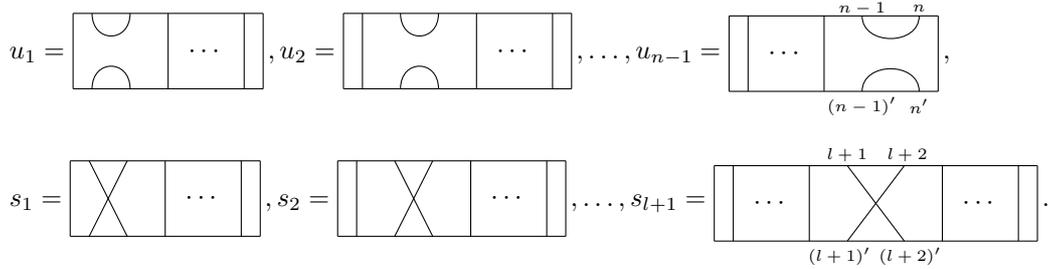
\begin{figure}[!h]
  \begin{align*}
    &u_1=
      \raisebox{-0.4\height}{\begin{tikzpicture}
    \draw (.25,0)--(.25,1) ; 
    \draw (.25,1)--(2.75,1) ; 
    \draw (2.75,1)--(2.75,0) ;
    \draw (.25,0)--(2.75,0) ; 
\draw (.5,1)..controls(.5,.6) and (1,.6)..(1,1) ; 
\draw (.5,0)..controls(.5,.4) and (1,.4)..(1,0); 
\draw (1.5,1)--(1.5,0) node at (2.0,.5){$\cdots$} ; 
\draw (2.5,1)--(2.5,0); 
\end{tikzpicture}}\, , 
    u_2=\raisebox{-0.4\height}{\begin{tikzpicture}
    \draw (.25,0)--(.25,1) ;
    \draw (.25,1)--(3.25,1); 
    \draw (3.25,1)--(3.25,0); 
    \draw (.25,0)--(3.25,0) ; 
    \draw (.5,1)--(.5,0); 
\draw (2.0,0)--(2.0,1) node at (2.5,.5)  {$\cdots$}; 
\draw (1,1)..controls(1,.6) and (1.5,.6)..(1.5,1) ;
\draw (1,0)..controls(1,.4) and (1.5,.4)..(1.5,0); 
\draw (3,1)--(3,0);
\end{tikzpicture}}
      \, , \ldots, 
    u_{n-1}=\raisebox{-0.5\height}{\begin{tikzpicture}
    \draw (.25,0)--(.25,1) ;
    \draw (.25,1)--(3.0,1); 
    \draw (3,1)--(3,0) ; 
    \draw (.25,0)--(3.0,0) ; 
\draw (.5,0)--(.5,1) node at (1,.5){$\cdots$};
\draw (1.5,0)--(1.5,1); 
\draw (2,1)..controls(2,.6) and (2.75,.6)..(2.75,1)
node at (2,-.2){\tiny$(n-1)^\prime$}  node at (2.75,-.2){\tiny$n^\prime$}; 
\draw (2,0)..controls(2,.4) and (2.75,.4)..(2.75,0)
node at (2,1.1){\tiny$n-1$} node at (2.75,1.1){\tiny$n$}; 
\end{tikzpicture}}\, ,
      \\
  &s_1=\raisebox{-0.4\height}{\begin{tikzpicture}
    \draw (.25,0)--(.25,1) ; 
    \draw (.25,1)--(2.75,1) ; 
    \draw (2.75,1)--(2.75,0) ; 
    \draw (.25,0)--(2.75,0) ; 
    \draw (.5,1)--(1.0,0); 
    \draw (1.0,1)--(.5,0); 
\draw (1.5,1)--(1.5,0) node at (2.0,.5){$\cdots$} ; 
\draw (2.5,1)--(2.5,0); 
\end{tikzpicture}}\, ,
    s_2=\raisebox{-0.4\height}{\begin{tikzpicture}
    \draw (.25,0)--(.25,1)  ;
    \draw (.25,1)--(3.25,1); 
    \draw (3.25,1)--(3.25,0) ; 
    \draw (.25,0)--(3.25,0) ; 
    \draw (.5,1)--(.5,0); 
    \draw (1,1)--(1.5,0);
    \draw (1.5,1)--(1,0);
\draw (2.0,0)--(2.0,1) node at (2.5,.5)  {$\cdots$}; 
\draw (3,1)--(3,0);
\end{tikzpicture}}\, , \ldots,
     s_{l+1}=\raisebox{-0.5\height}{\begin{tikzpicture}
    \draw (.25,0)--(.25,1) ; 
    \draw (.25,1)--(4.5,1)  ;
    \draw (4.5,1)--(4.5,0) ;
    \draw (.25,0)--(4.5,0) ;
\draw (.5,0)--(.5,1) node at (1,.5){$\cdots$};
    \draw (1.5,1)--(1.5,0); 
    \draw (2,1)--(2.75,0)  node at (2,1.15){\tiny$l+1$} node at (2.83,-.2){\tiny$(l+2)^\prime$};  
    \draw (2.75,1)--(2,0) node at (1.9,-.2){\tiny$(l+1)^\prime$} node at (2.79,1.15){\tiny$l+2$}; 
      \draw (3.25,1)--(3.25,0) node at (3.75,.5){$\cdots$} ; 
       \draw (4.25,1)--(4.25,0); 
  \end{tikzpicture}}\, .
 \end{align*}
   \caption{The generators for $J_{l,n}(\delta)$ }\label{generators of Jch5}
 \end{figure}

 We introduce some terminology for the diagrams and also a
 ``standard'' picture for a diagram $d$ with $m$ propagating lines.

\begin{defn} 
 We draw a diagram with the least number of total crossings. An arc is
\emph{covering} a crossing if the crossing is contained by the arc. I.e. when
we draw the diagram, the end points of the arc are either side of the
crossing and the crossing is above (respectively below) the arc if it
its a top (bottom) arc respectively. We will also say an arc
\emph{covers}
another arc if the end points of the second arc lie between the end
points of the first arc.
\end{defn}

By ``turning point'' in the following we mean that when the line is
drawn on the plane, a turning point is a point on that line where is there a horizontal
tangent to that line.

   We ``standardise'' the drawing for the diagram $d$ with $m$
  propagating lines and height $l$ as follows:

  Step 1: draw all arcs as ``simpler'' arcs, i.e. the lines only have one stationary point.

Step 2: draw all top arcs so they do not intersect any bottom
arcs.

Step 3: If an arc covers another arc, then we draw $d$ so that these
arcs do not intersect.

Step 4: pull propagating lines ``taut'', i.e. All
propagating lines now have no turning points in them.

Step 5: pull any crossings in propagating lines to the middle of the
diagram so that they are not covered by any arcs.

The diagram $d$ then has a top half-diagram, a permutation on the $m$
propagating lines ($\pi \in\sym_m$) in the middle and a bottom half-diagram.


Step 2 reduces crossings in the arcs, and leaves crossings with
propagating lines the same, so this does not produce any new
crossings. 
This also leaves a region in the middle of the
diagram that touches both sides of the diagram and contains no points
from arcs, although it may contain propagating lines (if $m >0$). (If
$m =0$ then this region is empty.)

In Step 5, 
the crossing number
of this crossing either stays the same or goes down with this
procedure. The total number of crossings in the propagating lines stays the same.

This procedure
either reduces the number of crossings of left-height $l$ in the
picture for the diagram $d$ or leaves it
the same, but it cannot increase it. In particular, it does not create
a crossing with a higher height than $l$.

NB: if standardisation removes all crossings of height $l$ in a
diagram, then that diagram did not have height $l$ to start with, as
we now have a picture for $d$ where all crossings are of height
smaller than $l$.

\begin{thm}\label{thm:genset}
 The algebra $J_{l,n}(\delta)$ with left height $l$ is generated by
 the diagrams $u_i$, $1\leq i\leq n-1$ and $s_m$, $1\leq m \leq l+1$
 with $l=\{0, 1, 2, \ldots,  n-1\}$, as in Figure \ref{generators of
   Jch5}.
\end{thm}
 \begin{proof}
We want to prove that the algebra $J_{l,n}(\delta)$ is generated by the set $A$
\begin{center}
$A:=\{u_i\}_{i=1}^{i=n-1}\cup \{s_m\}_{m=1}^{m=l+1}.$
\end{center}
It is clear that the set of diagrams which is generated by $A$ is
contained in $J_{l,n}(\delta)$ since the concatenation of these
diagrams can  never increase the left height $l$. So we need to prove
that every element in $J_{l,n}(\delta)$ can be written as a product of
elements belonging to $A$. Now the strategy of the proof is as follows:
\begin{itemize}
  \item Use the induction on the left-height $l$ of any diagram $d\in
    J_{l,n}(\delta)$ and on the number of features $r$ of left-height
    $l$
  \item ``standardise'' the drawing for the diagram $d$
  \item Divide the diagram $d$ into $d_1$, $s_{l+1}$ and $d_2$ such
    that $d_1$ and $d_2$ have less than $r$ features of left-height $l$
\end{itemize}
\begin{figure}[!h]
\begin{center}
  \begin{tikzpicture}
  \draw (0,0)--(0,1.5) node at (-.5,.75){$d=$};
     \draw (0,1.5)--(2.5,1.5) node at (3,.75){$\leadsto$};
     \draw (2.5,1.5)--(2.5,0);
     \draw (2.5,0)--(0,0);
     \fill [gray] (0,0) rectangle (2.5,1.5);
   \end{tikzpicture}
  \begin{tikzpicture}
    \draw (0,0)--(0,1.5)  node at (1.25,1.25){$d_1$};
     \draw (0,1.5)--(2.5,1.5) node at (1.25,.25){$d_2$} ;
     \draw (2.5,1.5)--(2.5,0) node at (1.25,.75){$s_{\tiny{l+1}}$};
     \draw (2.5,0)--(0,0);
     \draw (0,.5)--(2.5,.5);
     \draw (0,1)--(2.5,1);
      \draw(0,1.5)--(2.5,1.5);
  \end{tikzpicture}
   \end{center}
   \end{figure}

If $l=-1$, then the result follows from that for the Temperley-Lieb
algebra $\TL_n(\delta)$\cite[Section 6.1 and 6.3]{potts}.

Now suppose $d$ is any diagram in $J_{l,n}(\delta)$ with fixed left
height $l\in\{0, 1, 2, \ldots, n-2\}$.
Suppose $d$ is a
diagram with $r$ features of left height $l$ and $m$ propagating
lines.
  First, we ``standardise'' drawing for the diagram $d$ as above.
 Let $w_1$ be the top
  half-diagram with $c_1$ features of left height $l$, let $w_2$ be
  the bottom half-diagram with $c_2$ features of left height $l$ and
  $\pi$ is the middle of diagram $d$ with left height $l$.
This standardisation
either reduces the number of crossings of left-height $l$ in the
picture for the diagram $d$ or leaves it
the same, but it cannot increase it. In particular, it does not create
a crossing with a higher height than $l$.

$$\begin{tikzpicture}
    \draw (0,0)--(0,1.5)  node at (1.25,1.25){$w_1$};
     \draw (0,1.5)--(2.5,1.5) node at (1.25,.25){$w_2$} ;
     \draw (2.5,1.5)--(2.5,0) node at (1.25,.75){$\pi$};
     \draw (2.5,0)--(0,0);
     \draw (0,.5)--(2.5,.5);
     \draw (0,1)--(2.5,1);
      \draw(0,1.5)--(2.5,1.5);
  \end{tikzpicture}$$
If $r=0$, then we are done by induction.
For any $r\neq0$, we have the following cases:\\
 \noindent(1) There is a feature of height $l$ in $\pi$. I.e. a crossing of two propagating lines.\\
 \noindent(2) There is no crossing of height $l$ in $\pi$, but the
 crossing of left height $l$ is in $w_1$ or $w_2$.\\
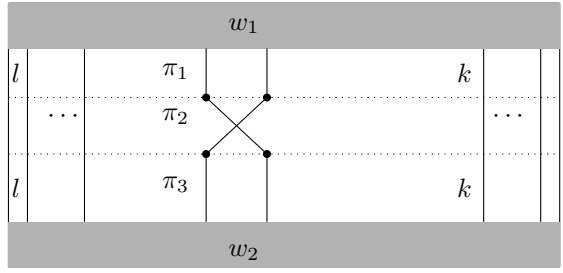
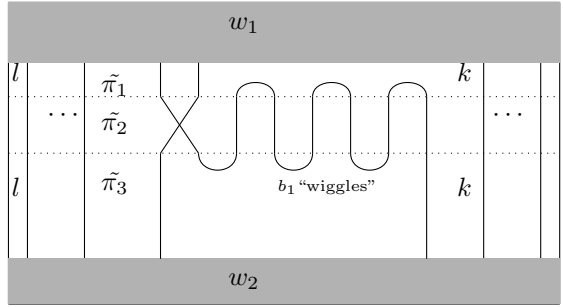
\begin{figure}[!h]
 \centering
\begin{subfigure}[b]{0.55\textwidth}  
  \begin{tikzpicture}
    \draw (0,.5)--(0,4)node at (.1,3.075){$l$};
    \draw (0,4)--(7.25,4)node at(6,3.075){$k$};
    \draw (7.25,4)--(7.25,.5)node at (.1,1.55){$l$};
    \draw (7.25,.5)--(0,.5)node at(6,1.55){$k$};
    \draw (.25,3.4)--(.25,1.1);
    \draw (1,3.4)--(1,1.1) node at (.75,2.5){$\cdots$};
    \draw (6.25,3.4)--(6.25,1.1) node at (6.6,2.5){$\cdots$};
     \draw (7,3.4)--(7,1.1)node at (2.2,3.1){$\pi_1$};
    \fill [color=gray!60](0,.5)rectangle(7.25,1.1);
    \node at (3.1,.7){$w_2$};

    \fill[color=gray!60](0,3.4)rectangle(7.25,4)   ;
     \node at (3.1,3.7){$w_1$};
   \draw [dotted](0,2.75)--(7.25,2.75) node at (2.2,2.5){$\pi_2$};
   \draw [dotted](0,2)--(7.25,2) node at (2.2,1.6){$\pi_3$};
\filldraw[black](2.6,2.75)circle(1.3pt);
\filldraw[black](3.4,2.75)circle(1.3pt);
\filldraw[black](2.6,2)circle(1.3pt);
\filldraw[black](3.4,2)circle(1.3pt);
\draw (2.6,2.75)--(3.4,2);
\draw (3.4,2.75)--(2.6,2);
\draw (2.6,2.75)--(2.6,3.4);
\draw (3.4,2.75)--(3.4,3.4);
\draw (2.6,2)--(2.6,1.1);
\draw (3.4,2)--(3.4,1.1);
 \end{tikzpicture}
 \subcaption{} \label{intersectiontwolines2}
\end{subfigure}
\begin{subfigure}[b]{0.55\textwidth}     
  \begin{tikzpicture}
    \draw (0,0)--(0,4)node at (.1,3.075){$l$};
    \draw (0,4)--(7.25,4)node at(6,3.075){$k$};
    \draw (7.25,4)--(7.25,0)node at (.1,1.55){$l$};
    \draw (7.25,0)--(0,0)node at(6,1.55){$k$};
    \draw (.25,3.2)--(.25,.6);
    \draw (1,3.2)--(1,.6) node at (.75,2.5){$\cdots$};
    \draw (6.25,3.2)--(6.25,.6) node at (6.6,2.5){$\cdots$};
     \draw (7,3.2)--(7,.6)  ;
      \fill [color=gray!60](0,0)rectangle(7.25,.6);
    \node at (3.1,.3){$w_2$};
    \fill[color=gray!60](0,3.2)rectangle(7.25,4)   ;
     \node at (3.1,3.7){$w_1$};
   \draw [dotted](0,2.75)--(7.25,2.75);
   \draw [dotted](0,2)--(7.25,2);
\draw (2,3.2)--(2,2.75)node at (1.4,2.9){$\tilde{\pi_1}$};
\draw (2,2.75)--(2.5,2)node at (1.4,2.4){$\tilde{\pi_2}$};
\draw (2.5,2.75)--(2,2);
\draw (2.5,2)..controls(2.5,1.7)and(3,1.7)..(3,2);
\draw (3,2)--(3,2.75) node at (1.4,1.6){$\tilde{\pi_3}$};
\draw (3,2.75)..controls(3,3)and(3.5,3)..(3.5,2.75);
\draw (3.5,2.75)--(3.5,2);
\draw (3.5,2)..controls(3.5,1.7)and(4,1.7)..(4,2);
\draw (4,2)--(4,2.75);
\draw (4,2.75)..controls(4,3)and(4.5,3)..(4.5,2.75);
\draw (4.5,2.75)--(4.5,2);
\draw (4.5,2)..controls(4.5,1.7)and(5,1.7)..(5,2);
\draw (5,2)--(5,2.75);
\draw (5,2.75)..controls(5,3)and(5.5,3)..(5.5,2.75);
\draw (5.5,2.75)--(5.5,.6);
\node at (4.2,1.55){\tiny$b_1\text{``wiggles''}$};
\draw (2,2)--(2,.6);
\draw (2.5,2.75)--(2.5,3.2);
 \end{tikzpicture} 
 \subcaption{}\label{intersectiontwolines}
\end{subfigure}
\caption{Case 1 the feature of height $l$ in $\pi$}\label{case3diagram}
\end{figure}

Case 1: There is a feature of height $l$ in $\pi.$ I.e. a crossing of two propagating lines.\\
We pull $\pi$ apart in the following way: Firstly we may draw $\pi$ so
that no two crossings lie on the same (invisible) horizontal line.
Then consider the crossing of height $l.$ We may take a small
horizontal strip containing this crossing and no other crossing. Then
we draw two dotted lines above and below the crossing of height $l$
giving this strip. This splits $\pi$ into a product of three
permutations $\pi=\pi_1\pi_2\pi_3\in S_m.$ We have $\pi_2$ is the
permutation $(l+1, l+2)$ in $S_m$ as it is of height $l.$ Since we did
not create any new crossings, we have $w_1\pi_1$ has one less crossing
than $w_1\pi_1\pi_2.$ Also $\pi_3w_2$ has one less crossing than
$\pi_2\pi_3w_2$ and the total number of crossing in $w_1\pi_1$ and
$\pi_3w_2$ is less than $r.$

Now to make $\pi_2$ into $s_{l+1}$ between the two dotted lines, we
need to get the right numbers of vertices and lines on the two dotted
lines. We need $n=l+k+2+2b_1,$ where $b_1$ is the number of wiggles in
Figure \ref{intersectiontwolines}. So we need $n$ to be the same
parity as $m$ which is true. The middle region $\tilde{\pi_2}$ that is
shown in Figure \ref{intersectiontwolines} is then $s_{l+1}.$ Then
$d_1$ is essentially $w_1\tilde{\pi_1}$ and has the same number of
features of height $l$ as in $w_1\pi_1$, $d_2$ is $\pi_3\tilde{w_2}$
and has the same number of features of height $l$ as in $\pi_3w_2.$
Thus $d_1\tilde{\pi_2}d_2$ gives us required decomposition for the
induction.
\begin{figure}[!h]
\centering
\begin{subfigure}[!h]{0.85\linewidth} 
\centering                                      
  \begin{tikzpicture}
  \draw (0,-1.5)rectangle (11.25,4);
  \draw (0,4)--(11.25,4);
  \draw(.25,4)--(.25,-.75);
  \draw (1.75,4)--(1.75,-.75);
  \node at (5.5,1.8){$\vdots$};
 \fill [color=gray!60] (0,3.5)rectangle (5.5,4);
 \draw(.25,4)--(.25,-.75);
  \draw (1.75,4)--(1.75,-.75);
 \draw (1,4)..controls(1,2.3)and(9.8,2.3)..(9.8,4);

 \draw (4,4)--(4,-.75)node at (2.75,1){$\cdots$};
 \draw (4.75,4)..controls(4.75,3)and (9,3)..(9,4);
 \draw (7,4)--(7,-.75);
 \fill [color=gray!60](7.5,4)..controls(7.5,3.3)and(8.5,3.3)..(8.5,4);
 \fill [color=gray!60](10,-.75)rectangle(11.25,4);
  \node at (10.6,2.8){$C$};
   \draw (.5,4)..controls(.5,.5)and(10.4,.5)..(10.4,4);
  \draw (1,4)..controls(1,2.3)and(9.8,2.3)..(9.8,4);

  \node at (8,3.8){$B$};
  \node at (3.25,3.7){$A$};
  \filldraw[black](7,3.3)circle(1.3pt);
  \node at (6.8,3){$l$};
  \fill [color=gray!40](0,-.75)rectangle(11.25,0);
  \node at (5.6,-.5){$\pi$};
  \fill [color=gray!20](0,-1.5)rectangle(11.25,-.75);
    \node at (5.6,-1){$w_2$};
    \draw [decorate,decoration={brace,amplitude=5pt,mirror,raise=.75ex}]
  (-.2,4) -- (-.2,0);
  \node at (-.75,1.78){$w_1$};
  \end{tikzpicture}\subcaption{}\label{intersection arcwithline}
\end{subfigure}  \vspace{.2cm}\break                        
\begin{subfigure}[!h]{0.75\linewidth} 
\centering                                   
  \begin{tikzpicture}
 \fill [color=gray!60] (0,3.5)rectangle (5.5,4);
 \draw (.5,4)--(.5,1.25);
 \draw (1,4)--(1,1.25);
 \draw (9.8,4)--(9.8,1.25);
 \draw (10.8,4)--(10.8,1.25);
 \node at (5.5,.2){$\vdots$};
  \draw (0,-2.25) rectangle (11.25,4);
  \draw (0,4)--(11.25,4)node at (-.25,3){$h_1$};
  \draw(.25,4)--(.25,-.75)node at (-.25,1.6){$h_2$};
  \draw (1.75,4)--(1.75,-.75)node at (-.25,.5){$h_3$};
 \draw (4,4)--(4,-.75)node at (2.4,1.6){$\cdots$};
 \draw (9,4)--(9,1.25);
 \draw (9,1.25)..controls(9,1)and(8.5,1)..(8.5,1.25);
 \draw(8.5,1.25)--(8.5,2);
 \draw (7.75,2)..controls(7.75,2.5)and(8.5,2.5)..(8.5,2);
 \draw(7.75,2)--(7.75,1.25);
 \draw (7.75,1.25)..controls(7.75,1)and(7,1)..(7,1.25);
 \draw (7,1.25)--(7,2);
 \draw (7,2)..controls(7,2.5)and(6.25,2.5)..(6.25,2);
 \draw(6.25,2)--(6.25,1.25);
 \draw (6.25,1.25)..controls(6.25,1)and(5.5,1)..(5.5,1.25);
 \draw(5.5,1.25)--(4.75,2);
 \draw (4.75,4)--(4.75,2);
 \node at  (7.4,3){\tiny$b_2\text{``wiggles''}$};
 \draw(7,4)..controls(7,3.6)and(5.5,2)..(4.75,1.25);
 \draw(4.75,1.25)--(4.75,-.75);
 \fill [color=gray!60](7.5,4)..controls(7.5,3.3)and(8.5,3.3)..(8.5,4);
 \fill [color=gray!60](10,-.75)rectangle(11.25,4);
 \draw (.5,1.25)..controls(.5,-.7)and(10.4,-.7)..(10.4,1.25);
 \draw (10.4,4)--(10.4,1.25);
  \node at (10.6,2){$C$};
  \draw (1,1.25)..controls(1,0)and(9.8,0)..(9.8,1.25);
  \node at (4.2,1.6){$l$};
  \node at (8,3.8){$B$};
  \node at (3.25,3.7){$A$};
  \draw [dotted](0,2)--(10,2);
  \draw [dotted](0,1.25)--(10,1.25);
  \fill [color=gray!40](0,-1.5)rectangle(11.25,-.75);
  \node at (5.6,-1){$\pi$};
  \fill [color=gray!20](0,-2.25)rectangle(11.25,-1.5);
    \node at (5.6,-2){$w_2$};
    \draw [decorate,decoration={brace,amplitude=5pt,mirror,raise=.5ex}]
  (.3,-.2) -- (3.9,-.2);%
  \node at ( 1.4,-.6){\tiny$m-k-1$};
  \node at (4.7,1.6){\tiny$x_1$};
   \filldraw[black](4.75,2)circle(1.3pt);
     \filldraw[black](9,2)circle(1.3pt);
       \filldraw[black](9,1.25)circle(1.3pt);
       \filldraw[black](5.5,2)circle(1.3pt);
        \filldraw[black](5.5,1.25)circle(1.3pt);
       \filldraw[black](4.75,1.25)circle(1.3pt);
  \end{tikzpicture}
  \subcaption{}\label{the third caseintersection twoarcs}
  \end{subfigure}
  \caption{Case 2a the feature of height $l$ in $w_1$ is the
    intersection an arc and a propagating line }\label{case2in the new
    proof}
\end{figure}
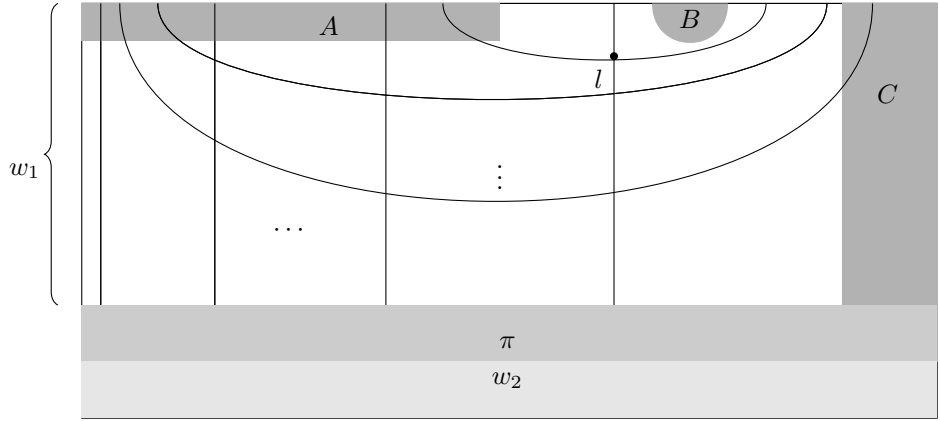
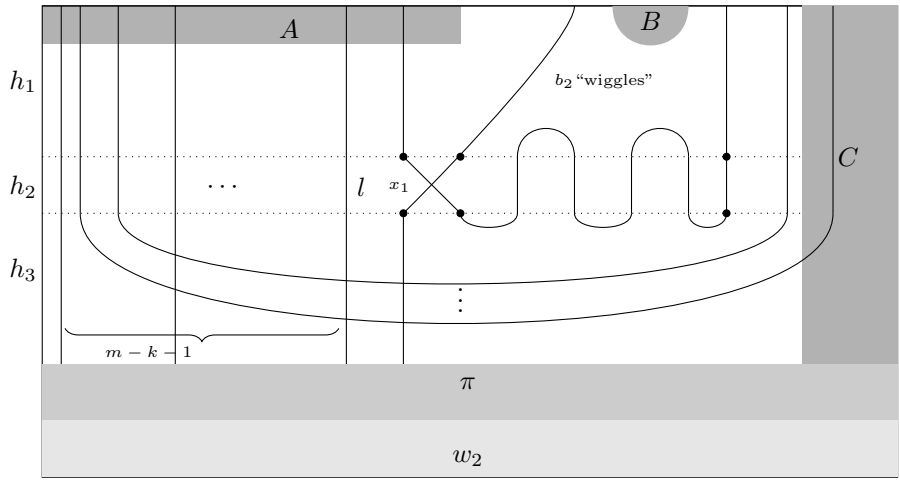\\
Case 2: There is no crossing in $\pi,$ but the crossing of left height $l$ is in $w_1$ or $w_2.$

Case 2a: The feature of left height $l$ is the intersection of an arc
and a propagating line.\\
 We assume without loss of generality that the arc is in the top
 half-diagram, $w_1$, the bottom case is similar. The region labelled
 $A$ in the Figure \ref{intersection arcwithline} can have arcs or
 propagating lines with crossing because the crossings in this region
 have left height less than or equal $l.$ The regions labelled $B$ in
 Figure \ref{intersection arcwithline} have only arcs without any
 crossing because any crossing this would create a crossing with left
 height $l+1.$ The region labelled $C$ in Figure \ref{intersection
   arcwithline} could have crossings of propagating lines or arcs such
 that the left height less than or equal $l.$

We push all covering arcs down without creating new crossings. Then we
pull the crossing of left height $l$ down and draw two dotted
horizontal lines above and below this crossing to get a single
crossing, it has left height $l$ so it looks like $s_{l+1}.$ Let
$k\geq 0$ be the number of propagating lines on the right of the
crossing of left height $l$ and suppose that $j$ is the number of
covering arcs.

We need to get the right number of vertices and lines on the two
dotted lines. There are $(m-k-1)+3+2j+2b_2+k$ vertices on each of the
dotted lines as shown in Figure \ref{the third caseintersection
  twoarcs}, where $b_2$ is the number of wiggles. We have
$m+2j+2+2b_2$ has the same parity of $n.$ Now we need to add
$2b_2=n-(m+2j+2)$ vertices to each of the dotted lines. We wiggle the
arc at position $l+1$ for $b_2$ times to get the right numbers of
vertices and vertical lines on each of the dotted lines. The middle
region $h_2$ that is shown in Figure \ref{the third caseintersection
  twoarcs} looks like $s_{l+1}.$ The diagram $h_2$ has one feature of
left height $l.$ We can decompose the diagram $w_1$ into a product of
three smaller diagrams $w_1=h_1h_2h_3$ with $c_1$ features of
left-height $l.$ The top diagram $h_1$ has strictly less features of
height $l$ than $w_1.$ Any crossing in region labelled $C$ stay in $C$
and these crossings are in $h_1$. The bottom diagram $h_3$ has all
covering arcs and any crossings of propagating lines with these
covering arcs have left height at most $l,$ the number of features of
height $l$ in $h_3$ is the same as $\pi w_2 < r.$ The total number of
crossings of left height $l$ in $h_1$ and $h_3$ is less than $r.$ 
Then $d_1$ is essentially $h_1$ has number of features of left-height
$l$ less than $r,$ $d_2=h_3\pi w_2$ and has the same number of
features in $w_2$ that is less than $r$ and the diagram $s_{l+1}=h_2$
has one feature of height $l.$
Thus $d_1s_{l+1}d_2$ gives us required decomposition for the induction.
 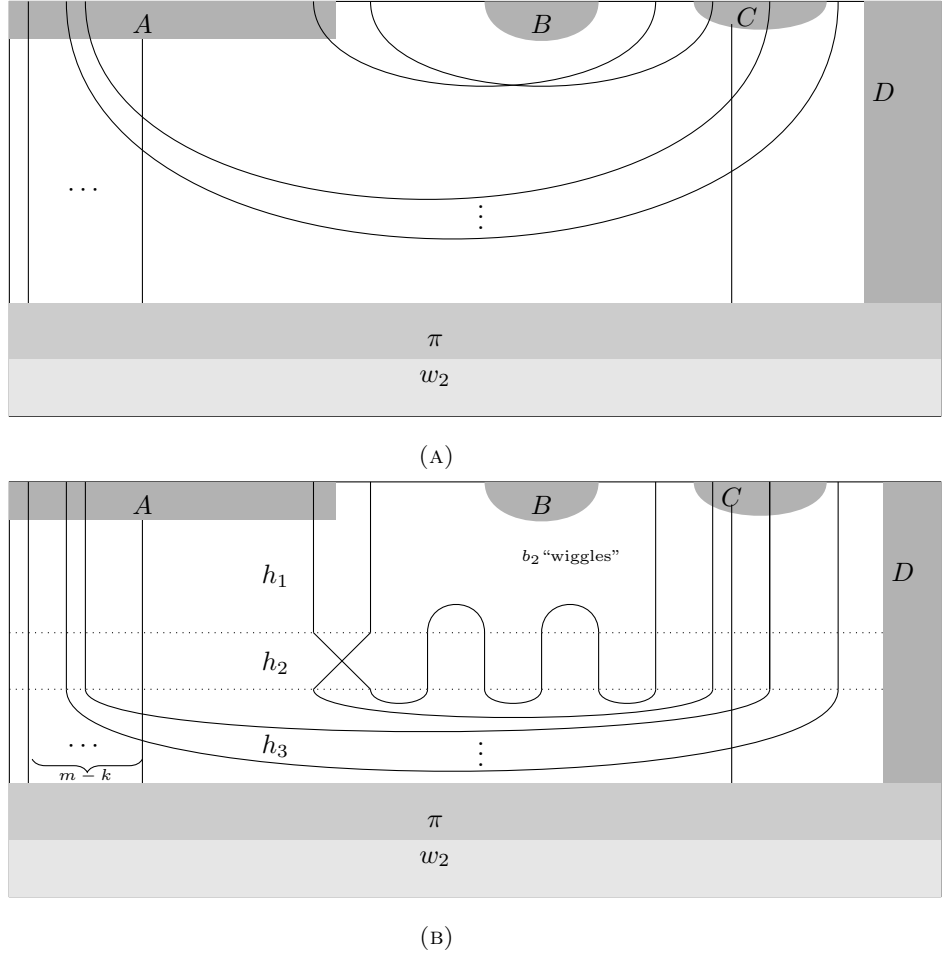
\begin{figure}[!htb]
\centering
\begin{subfigure}[!h]{0.75\linewidth} 
  \begin{tikzpicture}
  \draw (0,-1.5)rectangle (12.25,4);
  \draw (1.75,3.7)--(1.75,0);
  \fill [color=gray!60](0,3.5)rectangle(4.3,4);
  \node at (1.75,3.7){$A$};
  \fill [color=gray!60](11.5,0)rectangle(12.25,4);
  \node at (11.75,2.8){$D$};
  \fill [color=gray!60](11.25,0)rectangle(12.25,4);
  \node at (11.5,2.8){$D$};
 \fill [color=gray!60] (6.25,4)..controls(6.25,3.3)and(7.75,3.3)..(7.75,4);
  \node at(7,3.7){$B$};
  \fill [color=gray!60] (9,4)..controls(9,3.5)and(10.75,3.5)..(10.75,4);
  \node at (9.7,3.8){$C$};
  \draw (4,4)..controls (4,2.5)and (8.5,2.5)..(8.5,4);
  \draw (4.75,4)..controls(4.75,2.5)and(9.25,2.5)..(9.25,4);
  \draw (.75,4)..controls(.75,-.2)and(10.9,-.2)..(10.9,4);
  \draw (1,4)..controls(1,.5)and(10,.5)..(10,4);
  \node at (6.2,1.25){$\vdots$};

  \draw (.25,4)--(.25,0);
  \node at (1,1.5){$\cdots$};
  \draw (9.5,3.7)--(9.5,0);
  \fill [color=gray!40](0,-.75)rectangle(12.25,0);
  \node at (5.6,-.5){$\pi$};
  \fill [color=gray!20](0,-1.5)rectangle(12.25,-.75);
    \node at (5.6,-1){$w_2$};
  \end{tikzpicture}
  \subcaption{}\label{the third caseintersection twoarcs(2)}
  \end{subfigure}\break
  \begin{subfigure}[!h]{0.75\linewidth} 
  \begin{tikzpicture}
  \draw (0,-1.5)rectangle (12.25,4);
  \draw (1.75,3.7)--(1.75,0);
  \fill [color=gray!60](0,3.5)rectangle(4.3,4);
  \node at (1.75,3.7){$A$};
   \draw (8.5,4)--(8.5,1.25);
   \draw (5.5,2)--(5.5,1.25);
   \draw (6.25,2)--(6.25,1.25);
   \draw (7,2)--(7,1.25);
     \draw (7.75,2)--(7.75,1.25);
     \draw(8.5,1.25)..controls(8.5,1)and(7.75,1)..(7.75,1.25);
     \draw (7.75,2)..controls(7.75,2.5)and (7,2.5)..(7,2);
     \draw(7,1.25)..controls(7,1)and(6.25,1)..(6.25,1.25);
     \draw (6.25,2)..controls(6.25,2.5)and (5.5,2.5)..(5.5,2);
     \draw(5.5,1.25)..controls(5.5,1)and(4.75,1)..(4.75,1.25);
     \draw (4,2)--(4.75,1.25);
     \draw (4.75,2)--(4,1.25);
 \node at (6.2,.5){$\vdots$};
  \fill [color=gray!60](11.5,0)rectangle(12.25,4);
  \node at (11.75,2.8){$D$};
 \fill [color=gray!60] (6.25,4)..controls(6.25,3.3)and(7.75,3.3)..(7.75,4);
  \node at(7,3.7){$B$};
  \fill [color=gray!60] (9,4)..controls(9,3.4)and(10.75,3.4)..(10.75,4);
  \node at (9.5,3.8){$C$};
  \draw (4,4)--(4,2);
  \draw (10,4)--(10,1.25);
  \draw (.25,4)--(.25,0) node at (3.5,2.75){$h_1$};
\draw (9.5,3.7)--(9.5,0)node at (3.5,1.6){$h_2$};
 \draw [dotted](0,2)--(11.5,2)node at(3.5,.5){$h_3$};
  \draw [dotted](0,1.25)--(11.5,1.25);
  \node at (1,.5){$\cdots$};
\draw (4,1.25)..controls(4,.75)and(9.25,.75)..(9.25,1.25);
  \draw (.75,1.25)..controls(.75,-.2)and(10.9,-.2)..(10.9,1.25);
  \draw (1,1.25)..controls(1,.5)and(10,.5)..(10,1.25);
  \draw (1,4)--(1,1.25);
  \draw (10,4)--(10,1.25);
  \draw (10.9,4)--(10.9,1.25);
\draw (4.75,4)--(4.75,2);
  \draw (9.25,4)--(9.25,1.25);
  \draw (.75,1.25)--(.75,4);
\node at  (7.4,3){\tiny$b_2\text{``wiggles''}$};
  \fill [color=gray!40](0,-.75)rectangle(12.25,0);
  \node at (5.6,-.5){$\pi$};
  \fill [color=gray!20](0,-1.5)rectangle(12.25,-.75);
    \node at (5.6,-1){$w_2$};
     \draw [decorate,decoration={brace,amplitude=5pt,mirror,raise=.5ex}]
  (.3,.4) -- (1.75,.4);
  \node at (1,.1){\tiny$m-k$};
  \end{tikzpicture}
  \subcaption{}\label{the third caseintersection twoarcs(3)}
  \end{subfigure}
  \caption{Case 2b the feature of height $l$ in $w_1$ is the intersection of two arcs }\label{case3in the new proof}
\end{figure}

Case 2b: The feature of left height $l$ in $w_1$ is the intersection
of two arcs.\\
We assume without loss of generality that the arc is in the top
half-diagram, $w_1$, the bottom case is similar. We can deform the
arcs to look like Figure \ref{the third caseintersection twoarcs(3)}
by pulling the two arcs down (into the middle between two dotted
lines) without creating new crossings. We may assume the two arcs
themselves only cross each other once. The region labelled $A$ in the
Figure \ref{the third caseintersection twoarcs(2)} can have arcs or
propagating lines because they have left height less than or equal
$l.$ The region labelled $B$ has only arcs without any crossings
because any crossing this would create a crossing with left height
$l+1.$ The region labelled $C$ could have arcs or propagating lines
with crossings such that these crossings in the alcove of left height
at most $l.$ The region labelled $D$ may contain arcs or propagating
lines. Also the region $D$ could have crossings of height at most $l.$
Let $k$ be the number of propagating lines on the right of crossing of
height $l.$

Now we take a horizonal strip as in case 1 this gives two horizonal
dotted lines above and below the crossing it has left height $l$ so it
looks like $s_{l+1}.$ We push all the covering arcs down under the
dotted lines.

Now to create $h_2$ (see figure \ref{fig5}) between the two dotted lines, we need to get the
right number of vertices and vertical lines on two dotted lines. Let
$j$ be the number of covering arcs. There are $m-k$ propagating lines
on the left of the crossing. There are $(m-k)+k+2j+2b_3+4$ vertices on
each of the dotted lines as shown in Figure \ref{the third
  caseintersection twoarcs(3)}, where $b_3$ is the number of
wiggles. We have $m+2j+2b_3+4$ has the same parity of $n.$ Now we need
to add $2b_3=n-(m+2j+4)$ vertices to each of the dotted lines. We
wiggle the arc at position $l+1$ for $b_3$ times to get the right
numbers of vertices and vertical lines on each of the dotted
lines. The middle region $h_2$ that is shown in Figure \ref{the third
  caseintersection twoarcs(3)} looks like $s_{l+1}.$

 We have $h_2$ has one feature of left height $l.$ Any crossing in
 region labelled $D$ stay in $h_1.$ Then $w_1$ decompose to three
 small diagrams $w_1=h_1h_2h_3.$ The diagram $h_1$ has number of
 features of height $l$ less than $w_1,$ this include the crossing in
 region labelled $D.$ The diagram $h_3$ has all crossing of
 propagating lines with covering arcs. This bottom diagram has the
 same number of features of height as $\pi w_2<r.$ Then $d_1$ is
 essentially $h_1$ and the number of features of height $l<r,$
 $s_{l+1}=h_2$ and the diagram $d_2$ is $h_3\pi w_2$ and has number of
 features of height $l$ as $\pi w_2<r.$ The total number of features
 of height $l$ in $d_1$ and $d_3$ is less than $r.$ So we can
 decompose the diagram $d$ into a product of three diagrams with
 $d_1s_{l+1}d_2$ with the desired properties.
For example,\\
The diagram $d$ has $2$ features of left-height $l=2.$ Now to make the
diagram $d$ in standardisation, we push all arcs down under the dotted
lines and pull the two arcs in the region $h_2$. Suppose
$w_1=h_1h_2h_3$, then $w_1$ has $2$ features of left-height equal $2.$
The region $h_2$ should has one crossing of left-height $2$ and it
likes $s_{3}.$ The diagram $h_3$ has the same left-height equal $1$ as
in $\pi w_2$. The diagram $h_1$ has $1$ feature of left-height $2.$
Then the diagram $d_1=h_1$, $h_2=s_3$ and $d_2=h_3\pi w_2$. We can
decompose the diagram $d$ into product of three diagrams.
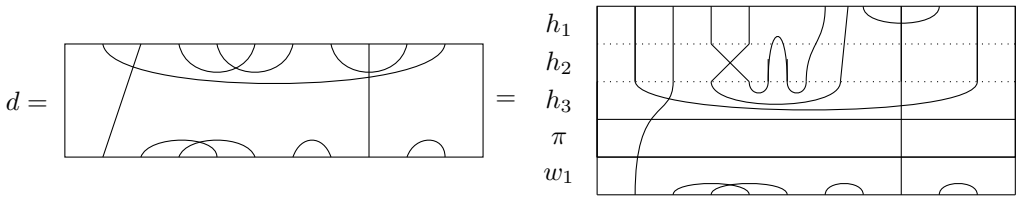
\begin{figure}[ht]
  \begin{tikzpicture}
  \node at (-.5,.75){$d=$};
  \draw (0,0) rectangle (5.5,1.5);
  \draw (1,1.5)--(.5,0);
  \draw (4,1.5)--(4,0);
  \draw (.5,1.5)..controls(.6,.8)and(4.9,.8)..(5,1.5);
  \draw (1.5,1.5)..controls (1.6,1) and (2.4,1)..(2.5,1.5);
  \draw (2,1.5)..controls (2.1,1) and (2.9,1)..(3,1.5);
  \draw (3.5,1.5)..controls (3.6,1) and (4.4,1)..(4.5,1.5);
  \draw (1,0)..controls (1.1,.3) and (2,.3)..(2,0);
   \draw (1.5,0)..controls (1.6,.3) and (2.4,.3)..(2.5,0);
   \draw (3,0)..controls (3.1,.3) and (3.4,.3)..(3.5,0);
   \draw (4.5,0)..controls (4.6,.3) and (5,.3)..(5,0);
   \node at (5.8,.75){$=$};
  \node at (6.5,1.75){$h_1$};
  \node at (6.5,1.25){$h_2$};
  \node at (6.5,.75){$h_3$};
  \node at (6.5,.25){$\pi$};
    \node at (6.5,-.25){$w_1$};
  \draw (7,0)rectangle (12.5,2);
  \draw (7,0)--(12.5,0);
   \draw (7,.5)--(12.5,.5);
 \draw [dotted] (7,1)--(12.5,1);
  \draw [dotted] (7,1.5)--(12.5,1.5);
  \draw (7,0)--(7,1);
  \draw (11,2)--(11,-.5);
\draw (8,-.5)..controls(8,-.3) and (9,-.3)..(9,-.5);
\draw (8.5,-.5)..controls(8.5,-.3) and (9.5,-.3)..(9.5,-.5);
\draw (10,-.5)..controls(10,-.3) and (10.5,-.3)..(10.5,-.5);
\draw (11.5,-.5)..controls(11.5,-.3) and (12,-.3)..(12,-.5);
\draw (7.5,2)--(7.5,1);
\draw (8,2)--(8,1);
\draw (12,2)--(12,1);
\draw (11,2)--(11,1);
\draw (7.5,2)--(7.5,1);
\draw (10.5,2)..controls(10.5,1.7) and (11.5,1.7)..(11.5,2);
\draw (7.5,1)..controls (7.6,.5)and (12,.5)..(12,1);
\draw (12,2)--(12,1);
\draw (8,1)..controls (8,.5)and (7.5,.8)..(7.5,-.5);
\draw (7,-.5)--(7,2);
\draw (7,-.5)--(12.5,-.5);
\draw (12.5,-.5)--(12.5,2);
\draw (8.5,2)--(8.5,1.5);
\draw (9,2)--(9,1.5);
\draw (8.5,1.5)--(9,1);
\draw (9,1.5)--(8.5,1);
\draw (9,1)..controls(9,.8)and(9.25,.8)..(9.25,1);
\draw (9.25,1)--(9.25,1.3);
\draw (9.25,1)..controls(9.25,1.8)and(9.5,1.8)..(9.5,1);
\draw (9.5,1.3)--(9.5,1);
\draw (9.5,1)..controls(9.5,.8)and(9.75,.8)..(9.75,1);
\draw (9.75,1)..controls (9.75,1.5)and (10,1.5)..(10,2);
\draw (8.5,1)..controls(8.5,.6)and(10.2,.6)..(10.2,1);
\draw (10.2,1)--(10.3,2);
   \end{tikzpicture}
\caption{creating $h_2$ in $d$.}\label{fig5}
\end{figure}
\end{proof}

\begin{exmp}
  Let $n=8$ and $l=3$. We can decompose the diagram $d$ below which
  has height $3$, into a product of generators. 
\begin{align*}
  d =
\raisebox{-12pt}{
     \begin{tikzpicture}
    \draw (0,0)--(0,1) ;
    \draw (0,1)--(4.5,1);
    \draw (4.5,1)--(4.5,0);
    \draw (4.5,0)--(0,0);
    \draw (.5,1)..controls (1,.6) and (3,.5)..(3.5,0);
    \draw (1,1)..controls (1.1,.6) and (.75,.5)..(.5,0);
    \draw (1.5,1)..controls (1.5,.7) and (3,.7)..(3,1);
    \draw (2.5,1)..controls (2.5,.7) and (4,.7)..(4,1);
    \draw (2,1)..controls(2,.75) and (1,.5)..(1,0);
    \draw (3.5,1)..controls (3.5,.6) and (2,.5)..(2,0);
    \draw (1.5,0)..controls (1.5,.3) and (3,.3)..(3,0);
     \draw (2.5,0)..controls (2.5,.3) and (4,.3)..(4,0);
    \end{tikzpicture}}
    &=
\raisebox{-197pt}{
    \begin{tikzpicture}
    \draw (0,0)--(0,14); 
    \draw (0,14)--(4.5,14);
  \draw (4.5,14)--(4.5,0);
  \draw (4.5,0)--(0,0);
  \draw (.5,14)--(.5,13) node at (5,13.5){$s_3$};
   \draw (1,14)--(1,13) ;
  \draw (1.5,14)--(2,13);
  \draw (2,14)--(1.5,13);
  \draw (2.5,14)--(2.5,13);
  \draw (3,14)--(3,13);
  \draw (3.5,14)--(3.5,13);
  \draw (4,14)--(4,13);
  \draw (0,13)--(4.5,13);
  \draw (.5,12)--(.5,13) node at (5,12.5){$s_4$};
  \draw (1,13)--(1,12);
  \draw (1.5,13)--(1.5,12);
  \draw (2,13)--(2.5,12);
  \draw (2.5,13)--(2,12);
\draw (3,13)--(3,12);
\draw (3.5,13)--(3.5,12);
\draw (4,13)--(4,12);
\draw  (0,12)--(4.5,12);
\draw (.5,12)--(.5,11) node at (5,11.5){$u_5$};
\draw (1,12)--(1,11);
\draw (1.5,12)--(1.5,11);
\draw (2,12)--(2,11);
\draw (2.5,12)..controls (2.5,11.7) and (3,11.7)..(3,12);
\draw (2.5,11)..controls (2.5,11.3) and (3,11.3)..(3,11);
\draw (3.5,12)--(3.5,11);
\draw (4,12)--(4,11);
\draw (0,11)--(4.5,11);
\draw (.5,11)--(1,10) node at (5,10.5){$s_1$};
\draw (1,11)--(.5,10);
\draw (1.5,11)--(1.5,10);
\draw (2,11)--(2,10);
\draw (2.5,11)--(2.5,10);
\draw (3,11)--(3,10);
\draw (3.5,11)--(3.5,10);
\draw (4,11)--(4,10);
\draw (0,10)--(4.5,10);
\draw (.5,10)--(.5,9) node at (5,9.5){$s_2$};
\draw (1,10)--(1.5,9);
\draw (1.5,10)--(1,9);
\draw (2,10)--(2,9);
\draw (2.5,10)--(2.5,9);
\draw (3,10)--(3,9);
\draw (3.5,10)--(3.5,9);
\draw (4,10)--(4,9);
\draw (0,9)--(4.5,9);
\draw (.5,9)--(.5,8) node at (5,8.5){$u_6$};
\draw (1,9)--(1,8);
\draw (1.5,9)--(1.5,8);
\draw (2,9)--(2,8);
\draw (2.5,9)--(2.5,8);
\draw (3,9)..controls (3,8.7) and (3.5,8.7)..(3.5,9);
\draw (3,8)..controls (3,8.3) and (3.5,8.3)..(3.5,8);
\draw (4,9)--(4,8);
\draw (0,8)--(4.5,8); 
\draw (.5,8)--(.5,7) node at (5,7.5){$s_4$};
\draw (1,8)--(1,7);
\draw (1.5,8)--(1.5,7);
\draw (2,8)--(2.5,7);
\draw (2.5,8)--(2,7);
\draw (3,8)--(3,7);
\draw (3.5,8)--(3.5,7);
\draw (4,8)--(4,7);
\draw (0,8)--(4.5,8);
\draw (0,7)--(4.5,7);
\draw (.5,7)--(.5,6) node at (5.25,6.5){$u_5u_7$};
\draw (1,7)--(1,6);
\draw (1.5,7)--(1.5,6);
\draw (2,7)--(2,6);
\draw (2.5,6)..controls (2.5,6.3) and (3,6.3)..(3,6);
\draw (2.5,7)..controls (2.5,6.7) and (3,6.7)..(3,7);
\draw (3.5,6)..controls(3.5,6.3) and (4,6.3)..(4,6);
\draw (3.5,7)..controls (3.5,6.7) and (4,6.7)..(4,7);
\draw (0,6)--(4.5,6);
\draw (.5,6)--(.5,5) node at (5,5.5){$s_3$};
\draw (1,6)--(1,5);
\draw (1.5,6)--(2,5);
\draw (2,6)--(1.5,5);
\draw (2.5,6)--(2.5,5);
\draw (3,5)--(3,6);
\draw (3.5,5)--(3.5,6);
\draw (4,5)--(4,6);
\draw (0,5)--(4.5,5);
\draw (.5,4)--(.5,5)node at (5,4.5){$u_6$};
\draw (.5,4)--(.5,5);
\draw (1,4)--(1,5);
\draw (1.5,4)--(1.5,5);
\draw (2,4)--(2,5);
\draw (2.5,4)--(2.5,5);
\draw (3,4)..controls(3,4.3) and (3.5,4.3)..(3.5,4);
\draw (3,5)..controls (3,4.7) and (3.5,4.7)..(3.5,5);
\draw (4,4)--(4,5);
\draw (0,4)--(4.5,4);
\draw (.5,3)--(.5,4) node at (5,3.5){$s_4$};
\draw (1,3)--(1,4);
\draw (1.5,3)--(1.5,4);
\draw (2,3)--(2.5,4);
\draw (2.5,3)--(2,4);
\draw (3,3)--(3,4);
\draw (3.5,3)--(3.5,4);
\draw (4,3)--(4,4);
\draw (0,3)--(4.5,3);
\draw (.5,2)--(.5,3) node at (5,2.5){$u_5$};
\draw (1,2)--(1,3);
\draw (1.5,2)--(1.5,3);
\draw (2,2)--(2,3);
\draw (2.5,2)..controls (2.5,2.3) and (3,2.3)..(3,2);
\draw (2.5,3)..controls (2.5,2.7) and (3,2.7)..(3,3);
\draw (3.5,2)--(3.5,3);
\draw (4,2)--(4,3);
\draw (0,2)--(4.5,2);
\draw (.5,1)--(.5,2) node at (5,1.5){$s_4$};
\draw (1,1)--(1,2);
\draw (1.5,1)--(1.5,2);
\draw (2,1)--(2.5,2);
\draw (2.5,1)--(2,2);
\draw (3,1)--(3,2);
\draw (3.5,1)--(3.5,2);
\draw (4,1)--(4,2);
\draw (0,2)--(4.5,2);
\draw(0,1)--(4.5,1);
\draw (.5,0)--(.5,1) node at (5,.5){$s_3$};
\draw (1,0)--(1,1);
\draw (1.5,0)--(2,1);
\draw (2,0)--(1.5,1);
\draw (2.5,0)--(2.5,1);
\draw (3,0)--(3,1);
\draw (3.5,0)--(3.5,1);
\draw (4,0)--(4,1);
    \end{tikzpicture} }\\
&=s_3s_4u_5s_1s_2u_6s_4u_5u_7s_3u_6s_4u_5s_4s_3
\end{align*}
\end{exmp}

\section{KMY algebras are semi-simple if $\delta$ not real}\label{semisimple}

\begin{thm}[{\cite[Theorem 48]{jacob1995linear}}]\label{eigenvalue}
  All eigenvalues of a real symmetric matrix are real.
\end{thm}

Let $\{v_i \}$ be a basis for $S^\lambda$.
Then 
we may choose the inner product on $S^{\lambda}$ so that it has real
values on the basis elements and so that $\langle v_i\mid v_j \rangle \in \R$ for $i\ne j$.
Indeed, the Specht modules are defined over $\Z$ and thus we can
choose to have integral values for the inner product between basis
elements. 

\begin{lem}\label{determin of real}
If $\delta$ is not real then the determinant of Gram matrix for the
cell module $\Delta_{l,n}(n-2,\lambda)$, $\lambda\vdash n$ is not zero.
\end{lem}
\begin{proof}
Let $G_{l,n}(n-2,\lambda)$ be the Gram matrix of 
$\Delta_{l,n}(n-2,\lambda)$. In computing the Gram matrix, we note
that it is only possible to get loops when the basis elements are
multiplied by themselves. ``Mixed'' products have either a real
value (coming from the inner product on the Specht module) or are
zero.

A basis element of $\Delta_{l,n}(n-1,\lambda)$ has one of two typical
types. Either it has one arc that can cross any of the first l-1
propagating lines, or is has an arc that does not cross any lines. Let
$v$ and $v'$ be basis elements in the Specht module
$S^\lambda$.
Elements of the half diagram basis for
$\Delta_{l,n}(n-1,\lambda)$ look like: 
$$
\begin{picture}(0,0)%
\includegraphics{basiselement1.pstex}%
\end{picture}%
\setlength{\unitlength}{3947sp}%
\begingroup\makeatletter\ifx\SetFigFont\undefined%
\gdef\SetFigFont#1#2#3#4#5{%
  \reset@font\fontsize{#1}{#2pt}%
  \fontfamily{#3}\fontseries{#4}\fontshape{#5}%
  \selectfont}%
\fi\endgroup%
\begin{picture}(1974,774)(739,-1573)
\put(1426,-1486){\makebox(0,0)[lb]{\smash{{\SetFigFont{10}{12.0}{\rmdefault}{\mddefault}{\itdefault}{\color[rgb]{0,0,0}$v$}%
}}}}
\end{picture}%

\qquad \raisebox{20pt}{\mbox{or}} \qquad
\begin{picture}(0,0)%
\includegraphics{basiselement2.pstex}%
\end{picture}%
\setlength{\unitlength}{3947sp}%
\begingroup\makeatletter\ifx\SetFigFont\undefined%
\gdef\SetFigFont#1#2#3#4#5{%
  \reset@font\fontsize{#1}{#2pt}%
  \fontfamily{#3}\fontseries{#4}\fontshape{#5}%
  \selectfont}%
\fi\endgroup%
\begin{picture}(1674,774)(439,-448)
\put(676,-361){\makebox(0,0)[lb]{\smash{{\SetFigFont{10}{12.0}{\rmdefault}{\mddefault}{\itdefault}{\color[rgb]{0,0,0}$v'$}%
}}}}
\end{picture}%

$$
A typical inner product then looks like:
$$
\begin{picture}(0,0)%
\includegraphics{innerproduct1.pstex}%
\end{picture}%
\setlength{\unitlength}{3947sp}%
\begingroup\makeatletter\ifx\SetFigFont\undefined%
\gdef\SetFigFont#1#2#3#4#5{%
  \reset@font\fontsize{#1}{#2pt}%
  \fontfamily{#3}\fontseries{#4}\fontshape{#5}%
  \selectfont}%
\fi\endgroup%
\begin{picture}(1974,1524)(739,-1573)
\put(1426,-1486){\makebox(0,0)[lb]{\smash{{\SetFigFont{10}{12.0}{\rmdefault}{\mddefault}{\updefault}{\color[rgb]{0,0,0}$v'$}%
}}}}
\put(1426,-241){\makebox(0,0)[lb]{\smash{{\SetFigFont{10}{12.0}{\rmdefault}{\mddefault}{\updefault}{\color[rgb]{0,0,0}$v$}%
}}}}
\end{picture}%

\quad \raisebox{40pt}{\mbox{ or }} \quad
\begin{picture}(0,0)%
\includegraphics{innerproduct2.pstex}%
\end{picture}%
\setlength{\unitlength}{3947sp}%
\begingroup\makeatletter\ifx\SetFigFont\undefined%
\gdef\SetFigFont#1#2#3#4#5{%
  \reset@font\fontsize{#1}{#2pt}%
  \fontfamily{#3}\fontseries{#4}\fontshape{#5}%
  \selectfont}%
\fi\endgroup%
\begin{picture}(2049,1524)(589,-1573)
\put(1426,-241){\makebox(0,0)[lb]{\smash{{\SetFigFont{10}{12.0}{\rmdefault}{\mddefault}{\updefault}{\color[rgb]{0,0,0}$v$}%
}}}}
\put(1426,-1486){\makebox(0,0)[lb]{\smash{{\SetFigFont{10}{12.0}{\rmdefault}{\mddefault}{\updefault}{\color[rgb]{0,0,0}$v'$}%
}}}}
\end{picture}%

$$
$$
\raisebox{40pt}{\mbox{ or }} \quad
\begin{picture}(0,0)%
\includegraphics{innerproduct3.pstex}%
\end{picture}%
\setlength{\unitlength}{3947sp}%
\begingroup\makeatletter\ifx\SetFigFont\undefined%
\gdef\SetFigFont#1#2#3#4#5{%
  \reset@font\fontsize{#1}{#2pt}%
  \fontfamily{#3}\fontseries{#4}\fontshape{#5}%
  \selectfont}%
\fi\endgroup%
\begin{picture}(2424,1524)(739,-2173)
\put(1351,-2086){\makebox(0,0)[lb]{\smash{{\SetFigFont{10}{12.0}{\rmdefault}{\mddefault}{\updefault}{\color[rgb]{0,0,0}$v'$}%
}}}}
\put(1351,-886){\makebox(0,0)[lb]{\smash{{\SetFigFont{10}{12.0}{\rmdefault}{\mddefault}{\updefault}{\color[rgb]{0,0,0}$v$}%
}}}}
\end{picture}%

\quad \raisebox{40pt}{\mbox{ or }} \quad
\begin{picture}(0,0)%
\includegraphics{innerproduct4.pstex}%
\end{picture}%
\setlength{\unitlength}{3947sp}%
\begingroup\makeatletter\ifx\SetFigFont\undefined%
\gdef\SetFigFont#1#2#3#4#5{%
  \reset@font\fontsize{#1}{#2pt}%
  \fontfamily{#3}\fontseries{#4}\fontshape{#5}%
  \selectfont}%
\fi\endgroup%
\begin{picture}(1749,1524)(589,-1573)
\put(1126,-1486){\makebox(0,0)[lb]{\smash{{\SetFigFont{10}{12.0}{\rmdefault}{\mddefault}{\updefault}{\color[rgb]{0,0,0}$v'$}%
}}}}
\put(1126,-286){\makebox(0,0)[lb]{\smash{{\SetFigFont{10}{12.0}{\rmdefault}{\mddefault}{\updefault}{\color[rgb]{0,0,0}$v$}%
}}}}
\end{picture}%

$$
and other variations.
Thus we can write the matrix $G_{l,n}(n-2,\lambda)$ as $M+\delta I_n$ where
and $I_n$ is the unit matrix and $M$ is a real matrix.

Now $\xi +\delta$ is an eigenvalue of $M+\delta I_n$ if and only
 if $\xi$ is an eigenvalue of $M$. Since the matrix $M$ is real and
 symmetric the eigenvalues of $M$ are real from Theorem
 \ref{eigenvalue}. Thus if $\delta$ is complex, $M+\delta I_n$ can have
 no real eigenvalues. So the determinant cannot be zero as no
 eigenvalue is zero.
\end{proof}

\begin{thm}\label{thm:semisimple}
For $n\geq 0$, if $\delta$ is not real then the algebra $J_{l,n}(\delta)$ is semisimple.
\end{thm}
\begin{proof}
Take $(n,\lambda)\in \Lambda_{n,2}^n$ and $ (n-2, \mu)\in
\Lambda_{n,2}^{n-2}.$
The module $\Delta_{l,n}(n,\lambda)$ is isomorphic to the Specht
module $S^\lambda$, $\lambda\vdash l+2$ which is simple in
characteristic zero. We consider
$\Delta_{l,n}(n-2,\mu),\mu\vdash\min\{n,l+2\}$. We
have from Lemma \ref{determin of real} that
$\det G_{l,n}(n-2,\mu)\neq0$ if $\delta$ is not real. So
$\Delta_{l,n}(n-2,\mu)$ is simple.
As $J_{l,n}(\delta)$ is quasi-hereditary these two modules are not
isomorphic. So we have
 \begin{equation*}
         \Hom\bigl(\Delta_{l,n}(n,\lambda),\Delta_{l,n}(n-2,\mu)\bigr)=0
        \end{equation*}
therefore by Theorem \ref{J_{l,n}semisimple}(ii) the algebra
$J_{l,n}(\delta)$ is semisimple.
\end{proof} 


\end{document}